\documentclass[oneside]{amsart}
\usepackage{amssymb}
\usepackage{dsfont}
\usepackage{hyperref}
\usepackage[a4paper]{geometry}

\theoremstyle{plain}
\newtheorem{thm}{Theorem}[section] 
\newtheorem{prop}[thm]{Proposition} 
\newtheorem{lem}[thm]{Lemma}

\theoremstyle{remark}
\newtheorem{rem}{Remark}[section]

\theoremstyle{definition}

\numberwithin{equation}{section}

\renewcommand*{\div}{\operatorname{div}}
\newcommand*{\curl}{\operatorname{curl}}

\newcommand*{\supp}{\operatorname{supp}}

\usepackage[textsize=tiny,textwidth=27mm]{todonotes}

\newcommand*{\bydef}{\overset{\rm def}{=}}
\newcommand*{\norm}[1]{\left\Vert #1\right\Vert}

\begin{document}
\title[Stability analysis of two-dimensional ideal flows]{Stability analysis of two-dimensional ideal flows with applications to viscous fluids and plasmas}

\author{Diogo Ars\'enio}
\author{Haroune Houamed}
\address{New York University Abu Dhabi \\
Abu Dhabi \\
United Arab Emirates} 
\email{\href{mailto:diogo.arsenio@nyu.edu}{diogo.arsenio@nyu.edu}, \href{mailto:haroune.houamed@nyu.edu}{haroune.houamed@nyu.edu}}


\date{\today}

\begin{abstract}
	We are interested in the stability analysis of two-dimensional incompressible inviscid fluids. Specifically, we revisit a known recent result on the stability of Yudovich's solutions to the incompressible Euler equations in $L^\infty([0,T];H^1)$ by providing a new approach to its proof based on the idea of compactness extrapolation and by extending it to the whole plane.
	
	This new method of proof is robust and, when applied to viscous models, leads to a remarkable logarithmic improvement on the rate of convergence in the vanishing viscosity limit of two-dimensional fluids. Loosely speaking, this logarithmic gain is the result of the fact that, in appropriate high-regularity settings, the smoothness of solutions to the Euler equations at times $t\in [0,T)$ is strictly higher than their regularity at time $t=T$. This ``memory effect'' seems to be a general principle which is not exclusive to fluid mechanics. It is therefore likely to be observed in other setting and deserves further investigation.
	
	Finally, we also apply the stability results on Euler systems to the study of two-dimensional ideal plasmas and establish their convergence, in strong topologies, to solutions of magnetohydrodynamic systems, when the speed of light tends to infinity. The crux of this asymptotic analysis relies on a fine understanding of Maxwell's system.
\end{abstract}

\maketitle

\tableofcontents

\newpage

\section{Stability of the incompressible Euler system in the plane}  \label{sec.stability Euler}

\subsection{Introduction and first main result}

We are interested in the stability of the two-dimensional incompressible Euler equations 
\begin{equation}\label{Euler*}
	\begin{cases}
		\begin{aligned}
 \partial_t u^{\varepsilon} +u^{\varepsilon} \cdot\nabla u^{\varepsilon}    + \nabla p^{\varepsilon} &=   g^{\varepsilon}, 
			\\
			 \div u^{\varepsilon} &= 0,\\
			{ u^{\varepsilon}}_{|_{t=0}} &=u_{0}^{\varepsilon},
		\end{aligned}
	\end{cases}
\end{equation}
in the limit $\varepsilon\rightarrow 0$, where $u^\varepsilon=u^\varepsilon(t,x)$ are fluid velocity fields and $g^\varepsilon=g^\varepsilon(t,x)$ are prescribed source terms, with $(t,x)\in \mathbb{R}^+\times\mathbb{R}^2$, taking values in $\mathbb{R}^2$. The pressures $p^\varepsilon$ are real-valued and also depend on $t$ and $x$.

Formally, if the inputs $( g^\varepsilon,u_{0}^{\varepsilon})$ converge   to some states $(g,u_0)$, as $\varepsilon$ vanishes, then it is expected that $u^\varepsilon$ approaches $u$, the solution of the same Euler equations
\begin{equation}\label{Euler}
	\begin{cases}
		\begin{aligned}
 \partial_t u +u  \cdot\nabla u  + \nabla p  &=   g , 
			\\
			 \div u  &= 0,\\
			{ u }_{|_{t=0}} &=u_{0 }.
		\end{aligned}
	\end{cases}
\end{equation} 
The celebrated Yudovich's theorem allows us to formulate the rigorous weak convergence result described in the following statement.

(Note that we employ below the classical notation $H^s$ to denote the Sobolev space of functions $f\in L^2$ such that the Riesz potential $|D|^sf$ also belongs to $L^2$, whereas $W^{1,p}$ denotes the Sobolev space of functions $f\in L^p$ such that $\nabla f\in L^p$. Homogeneous versions of these spaces are defined in a similar fashion.)

\begin{thm}[\cite{Yudovich1}]\label{yudovich:thm}
	Let $\varepsilon \in (0,1)$ and $T\in \mathbb{R}^+\cup \{ \infty\}$.
	For any initial data and source satisfying
	\begin{equation*}
		\begin{aligned}
			u_{0}^{\varepsilon}\in H^1 (\mathbb{R}^2),
			&&&
			\curl u_{0}^{\varepsilon} \in   L^\infty(\mathbb{R}^2),
			\\
			g^{\varepsilon }\in L^1([0,T];H^1 (\mathbb{R}^2)),
			&&&
			\curl g^{\varepsilon }\in L^1([0,T];  L^\infty (\mathbb{R}^2)),
		\end{aligned}
	\end{equation*}
	uniformly in $\varepsilon$,
	there is a unique solution $u^\varepsilon$   to \eqref{Euler*} on $[0,T]$  enjoying the bounds
	$$u^{\varepsilon }\in C ([0,T]; H^1 (\mathbb{R}^2)) , \quad \curl u^{\varepsilon }\in  L^\infty ([0,T];  L^\infty(\mathbb{R}^2)) ,$$
	uniformly in $\varepsilon.$
	Moreover, one has,  for all $t\in [0,T]$, that 
	\begin{equation}\label{Transport:ES}
		\norm {\curl u^\varepsilon(t)}_{L^p} \leq  \norm {\curl u^\varepsilon_0}_{L^p} + \norm {\curl  g^\varepsilon}_{L^1([0,t);L^p)}, \quad \text{for all } p\in [1,\infty].
	\end{equation}
	Furthermore, if $(u_0,g) $ denotes a weak limit, in the sense of distributions, of the family $( u^\varepsilon_0,g^\varepsilon)_{\varepsilon>0}$, then $u^\varepsilon$  converges weakly, in the sense of distributions,   to the unique solution $u$ of the Euler system \eqref{Euler}.
\end{thm} 

We are now interested in refining our understanding of the limit     of $u^\varepsilon$ to $u$, provided   the inputs of \eqref{Euler*} converge in a suitable strong sense.  Our first main result is featured in the following theorem, which is a variant of Theorem 1 from \cite{CDE} extended to an unbounded domain and more general source terms.

\begin{thm}\label{thm2} Let $T\in\mathbb{R}^+$ and $(u_0^\varepsilon,g^\varepsilon)_{\varepsilon>0}$ be a family of initial data and source terms satisfying the assumptions of Theorem \ref{yudovich:thm},  which converges to  $ (u_0,g)$   strongly in $H^1(\mathbb{R}^2)\times L^1([0,T];H^1(\mathbb{R}^2))$, as $\varepsilon\rightarrow 0$.  Then, it holds that
$$
\lim_{\varepsilon\rightarrow 0}  \sup_{t\in [0,T]}\norm {u^\varepsilon(t) - u (t) } _{H^1(\mathbb{R}^2)} =0,
$$
where $u^\varepsilon$ and $u$ are the unique solutions of \eqref{Euler*} and \eqref{Euler}, respectively.
\end{thm}
\begin{rem}
Observe in the above theorem that the convergence holds in the whole space $\mathbb{R}^2$.   Moreover,  employing  the Gagliardo--Nirenberg and H\"{o}lder inequalities,  note that one can show the convergence of the velocity in $L^\infty_t(  L^\infty \cap  W  ^{1,p})$, for all $p\in [2,\infty)$.  Accordingly, Theorem \ref{thm2} provides   a  convergence result  in spaces similar to those given in Theorem \ref{yudovich:thm}. 
\end{rem}

The novelty of our work lies in the method of proof of Theorem \ref{thm2}, given in Section \ref{subsection.THM2}, below, and in applications of this method in Sections \ref{section:rate} and \ref{section:plasma}, later on, where we give our second and third main results in Theorems \ref{thm-rate-general} and \ref{thm-EM}, respectively.

It is to be emphasized that our approach also allows us to extend Theorem \ref{thm2} to an inviscid limit. More precisely, considering the Navier--Stokes system
\begin{equation}\label{NS*}
	\begin{cases}
		\begin{aligned}
 \partial_t u^{\varepsilon} +u^{\varepsilon}  \cdot\nabla u^{\varepsilon}   - \varepsilon \Delta u^{\varepsilon}   + \nabla p^{\varepsilon}  &=   g^{\varepsilon} , 
			\\
			 \div u^{\varepsilon}  &= 0,\\
			{ u^{\varepsilon} }_{|_{t=0}} &=u_{0}^{\varepsilon} ,
		\end{aligned}
	\end{cases}
\end{equation}
we have the next result, whose proof is a straighforward adaptation of the proof of Theorem \ref{thm2}.

\begin{thm}\label{thm3}
	Let $T\in\mathbb{R}^+$ and $(u_0^\varepsilon,g^\varepsilon)_{\varepsilon>0}$ be a family of initial data and source terms satisfying the assumptions of Theorem \ref{yudovich:thm},  which converges to  $ (u_0,g)$   strongly in $H^1(\mathbb{R}^2)\times L^1([0,T];H^1(\mathbb{R}^2))$, as $\varepsilon\rightarrow 0$.  Then, it holds that
$$
\lim_{\varepsilon\rightarrow 0}  \sup_{t\in [0,T]}\norm {u^\varepsilon(t) - u (t) } _{H^1(\mathbb{R}^2)} =0,
$$
where $u^\varepsilon$ and $u$ are the unique solutions of \eqref{NS*} and \eqref{Euler}, respectively.
\end{thm}

There are several results in the literature investigating the inviscid limit of the Navier--Stokes system in the whole space $\mathbb{R}^2$ or the torus $\mathbb{T}^2$. The first results about such convergence   were established  in the case of strong solutions in \cite{BM81,C86,K72,S71}.

These results have then been further studied in rougher cases, where the convergence is established in the energy space. See  \cite{CW95,CW96} for the case vortex patches, and  \cite{CH96CPDE} where the velocity is not necessary  lipschitz and the rate of convergence deteriorates over time. Furthermore, optimal rates of convergence for velocities in vortex patch configurations have been established in \cite{AD04, M07}.

We also refer to \cite{BT13} for a survey of results, and a study of the inviscid regime in the weak setting of dissipative solutions.

More recently, there has been progress on the strong inviscid limit of vorticities  in Lebesgue spaces within Yudovich's class of solutions, see \cite{CCS,CDE,HCE}. Therefore, let us now clarify the novelty in our result as well as the key differences with other works.

\subsection{Overview of previous results}

We first note that the convergence in $L^\infty_tL^2(\mathbb{R}^2)$ of the velocity $u^{\varepsilon}$ to a solution of Euler's equations  is   standard and follows from classical  stability estimates based on the uniform bounds $\curl u^{\varepsilon} \in L^\infty_{t,x}$ first established by Yudovich \cite{Yudovich1}.

The corresponding convergence in $ L^\infty_t \dot{H}^1(\Omega)$, with $\Omega=\mathbb{T}^2$, has been established much more recently by  Constantin, Drivas and Elgindi  in  \cite{CDE}.  Their method of   proof  consists in approaching both systems \eqref{Euler} and  \eqref{NS*}  by two linear transport equations with regularized inputs, where the solutions are advected by the   velocities  $ u$ and $u^{\varepsilon}$, respectively.

Hence, the problem of proving the convergence of the velocity in $ L^\infty_t \dot{H}^1(\Omega)$ (or equivalently the vorticity in $ L^\infty_t L^2(\Omega)$) is then  reduced to proving a stability result for the regularized systems.  In    \cite{CDE}, in order to study the  stability      of the vortices in $L^\infty_tL^2$ in the asymptotic regime $\varepsilon\rightarrow 0$, the authors  perform an $L^2$-energy estimate in the vorticity formulation of the regularized linear systems. They are then able to close the estimates by   relying on two essential lemmas.

 The first one \cite[Lemma 1]{CDE} is, roughly speaking, a variation of the John--Nirenberg inequality, which states that, for any $f\in BMO(\mathbb{R}^2)$ and any compact domain $K\subset\mathbb{R}^2$, there exists $\beta>0$ and $C>0$ such that
$$ \int_{K}\exp ( \beta |f(x)|) dx \leq C.  $$
In \cite{CDE}, the authors utilize a variant of the preceding  inequality for $f=\nabla u$, which belongs to $BMO$ as soon as $\nabla\times u$ belongs to $L^\infty$.

 The second crucial ingredient in the proof given in \cite{CDE} establishes a loss estimate for the regularized vortices in
 $$L^\infty\left([0,T];\dot{W}^{1,p(t)}\right),$$
 where $p(t)$ is a decreasing continuous function of time. However, this bound is not uniform with respect to the regularizing parameter in the inputs.

Although the proof from  \cite{CDE} hinges on the boundedness of the torus $\mathbb{T}^2$, we emphasize that a careful and suitable adaptation of that proof would lead to similar results in the whole plane $\mathbb{R}^2$.  Rather than rigorously justifying the extension of the arguments from \cite{CDE} to the whole domain $\mathbb{R}^2$, we choose  to present hereafter a different approach, which does not rely on properties of $BMO$ functions or time-dependent Sobolev norms. Furthermore, this new approach will then play a key role in the proofs of the important applications of Sections \ref{section:rate} and \ref{section:plasma}, below.

Finally, we also refer to  \cite{CCS,HCE} for different proofs using techniques from the celebrated work of DiPerna and Lions \cite{DL89} on renormalized solutions of general transport equations.

\subsection{A new method of proof}

Our proof of Theorem \ref{thm2} establishes the $L^\infty_t\dot{H}^1$ stability of velocity fields without studying the equations satisfied by the difference $ u^\varepsilon-u$ in $L^\infty_t\dot{H}^1$.

Instead, we exploit the convergence of velocities in $L^\infty_tL^2$, which, as previously emphasized, is a classical consequence of the stability estimates established by Yudovich \cite{Yudovich1}, in combination with a time-dependent version of the following simple but essential extrapolation lemma. This result provides a useful criterion which allows us to recover strong compactness properties in endpoint functional settings. In particular, in the notation of the lemma, we will be using the functional spaces $\dot H^{s_0}=L^2$ and $\dot H^{s_1}=\dot H^{1}$.

\begin{lem}[Compactness Extrapolation Lemma]\label{lemma:extrapolation}
	Fix the dimension $d\geq 1$. Let $s_0<s_1$ be two real numbers and $(u^\varepsilon)_{\varepsilon\in (0,1]}$ be  a family of    bounded functions in  $ \dot{H}^{s_0} \cap \dot{H}^{s_1}(\mathbb{R}^d)$, uniformly in $\varepsilon\in (0,1]$. Further assume that
	$$u^\varepsilon\to u
	\quad\text{in } \dot{H}^{s_0}(\mathbb{R}^d).$$
	Then, it holds that
	$$u^\varepsilon\to u
	\quad\text{in } \dot{H}^{s_1}(\mathbb{R}^d)$$
	if and only if
	$$\lim_{\varepsilon \rightarrow 0} \norm {   \mathds{1}_{|D|\geq \Theta_\varepsilon}u^\varepsilon }_{\dot{H}^{s_1}}=0,$$ 
	for some $\Theta_\varepsilon>0$ satisfying
	\begin{equation*}
	  \lim_{\varepsilon \rightarrow 0}  \Theta_\varepsilon  = \infty
	  \quad \text{and}\quad
	  \lim_{\varepsilon \rightarrow 0}
	  \Theta_\varepsilon^{s_1-s_0} \norm { u^\varepsilon-u }_{\dot{H}^{s_0}}
	  =0.
	\end{equation*}
\end{lem}

\begin{proof}
The proof of the sufficiency of the criterion straightforwardly follows from the observations that
$$  \norm { \mathds{1}_{|D|\leq \Theta_\varepsilon}(u^\varepsilon-u) }_{\dot{H}^{s_1}} \leq \Theta_\varepsilon^{s_1-s_0} \norm { u^\varepsilon-u }_{\dot{H}^{s_0}}, $$
the fact that the right-hand side above vanishes due to the assumptions on $\Theta_\varepsilon$, and the convergence
$$\lim_{\varepsilon \rightarrow 0} \norm {   \mathds{1}_{|D|\geq \Theta_\varepsilon}u }_{\dot{H}^{s_1}}=0,$$
for any fixed $u\in \dot H^{s_1}$.

On the other hand, the justification of the direct implication in the lemma follows directly by writing   
\begin{equation*}
	\norm {   \mathds{1}_{|D|\geq \Theta_\varepsilon}u^\varepsilon }_{\dot{H}^{s_1}} \leq \norm {      u^\varepsilon-u }_{\dot{H}^{s_1}} + \norm {   \mathds{1}_{|D|\geq \Theta_\varepsilon}u }_{\dot{H}^{s_1}}
\end{equation*}
and noting that the right-hand side vanishes in the limit $\varepsilon\to 0$, which concludes the proof.
\end{proof}

Thus, by virtue of the compactness criterion given in the previous lemma, our work will be reduced to the control of some high frequencies of the vorticities $\omega^\varepsilon\bydef \curl u^\varepsilon$ in $L^\infty_t L^2$. This is crucial and will be achieved by performing an $L^2$-energy estimate on dyadic blocks of $\omega^\varepsilon$ which is compatible with the nonlinear structure of the equations and relies on the remarkable identities established in Lemma \ref{lemma:iden:1}, later on.

\subsection{Notation}

Before moving on to the proofs of our theorems, allow us to clarify some elements of notation that we are about to use.

We denote by $C>0$ any universal constant that is independent of the main variables of the given	 problem. Accordingly, we use the inequality $ A\lesssim B$ when there exists a constant $C>0$ independent of $A$ and $B$ such that $A\leq C B$. In general, the constant $C$ is allowed to change from one line to the next.

Moreover, for any bounded function $ m:\mathbb{R}^2 \rightarrow \mathbb{R}$, the Fourier multiplier operator $m(D)$ is defined by 
$$m(D)  f \bydef \mathcal{F}^{-1}\big(m(\xi)  \mathcal{F}f(\xi)\big),$$
for any tempered distribution $f\in \mathcal{S}'(\mathbb{R}^2)$,
where $\mathcal{F}$ and $\mathcal{F}^{-1}$ denote the Fourier transform and its inverse, respectively.

Finally, the commutator between two operators $Q$ and $S$ is denoted by $[Q,S]$ and is defined by the relation
$$[Q,S]\omega \bydef Q S\omega - SQ\omega,$$
for any suitable $\omega$.

\subsection{Proof of Theorem \ref{thm2}}\label{subsection.THM2}

We proceed in several steps:
\begin{enumerate}
	\item First, in Section \ref{section:recall}, we recall the classical ideas leading to the convergence $u^\varepsilon \to u$ in $L^\infty_tL^2$.
	\item Then, in Section \ref{section:low:freq}, we discuss the convergence of some suitable low frequencies of vorticities with a simple estimate based on the convergence of velocities in $L^\infty_tL^2$.
	\item Finally, in Section \ref{section:high:freq}, in the spirit of the Compactness Extrapolation Lemma (Lemma \ref{lemma:extrapolation}), we identify and control the remaining high frequencies of vorticities, thereby establishing their convergence and completing the proof of Theorem \ref{thm2}.
	\item The remaining sections, i.e., Sections \ref{section:energy:identity} and \ref{section:estimate:J1}, are dedicated to essential technical results which are justified separately, for the sake of clarity.
\end{enumerate}

\subsubsection{Convergence in $L^\infty_tL^2$}\label{section:recall}

The stability of the two-dimensional Euler equations in $L^\infty_tL^2$ is classical and follows from the uniqueness methods employed in \cite{Yudovich1}. We also refer to \cite[Lemma 4]{CDE} for a different proof in the torus which can be adapted to the whole plane. 

More specifically, it is possible to show that if
$$ \widetilde{u}^\varepsilon \bydef u^\varepsilon-u
\quad\text{and}\quad
\widetilde{g}^\varepsilon \bydef g^\varepsilon -g$$
are small in the sense that
\begin{equation*}
	\norm { \widetilde{u}^{\varepsilon } _0  }_{L^2}  +  \norm { \widetilde{g}^{\varepsilon }   }_{L^1([0,T];L^2)}
	< C_*e^{-\exp(C_*T)},
\end{equation*}
where
\begin{equation}\label{C*:def}
C_* \bydef C\sup_{\varepsilon>0} \left( \norm {(\omega_0,\omega^\varepsilon_0)}_{L^2\cap L^\infty}  + \norm {(\curl g,\curl g^\varepsilon)}_{L^1([0,T];L^2\cap L^\infty)}\right)  ,
\end{equation}  
for some universal constant $C>0$, then one has the stability estimate
\begin{equation}\label{u:CV:final0}
	\begin{aligned}
		\norm { \widetilde{u}^{\varepsilon }   }_{L^\infty ([0,T]; L^2)}  \lesssim_{C_*,T}\Big(  \norm { \widetilde{u}^{\varepsilon } _0  }_{L^2}  +  \norm { \widetilde{g}^{\varepsilon }   }_{L^1([0,T];L^2)}   \Big) ^{\exp (-C_*T)}.
	\end{aligned}
\end{equation}

For the sake of completeness, we now outline the idea leading to \eqref{u:CV:final0} by following the modern proof given in \cite[Section 7.3.3]{bcd11}. To that end, we begin by observing that $ \widetilde{u}^\varepsilon$ and $ \widetilde{g}^\varepsilon$
solve the system
\begin{equation*}
	\begin{cases}
		\begin{aligned}
 \partial_t \widetilde{u}^{\varepsilon }+u^{\varepsilon } \cdot \nabla \widetilde{u}^{\varepsilon }  + \nabla \widetilde{ p}^{\varepsilon } &= - \widetilde{u}^{\varepsilon } \cdot \nabla u+   \widetilde{g}^{\varepsilon }, 
			\\
			 \div \widetilde{u}^{\varepsilon }&= 0,\\ 
			   {\widetilde{u}}^{\varepsilon }_{|_{t=0}} &=u _ 0^{\varepsilon }  - u_0.
		\end{aligned}
	\end{cases}
\end{equation*}
Then, performing a standard $L^2$-energy estimate yields, for any $t\in [0,T]$,
\begin{equation*}
	\begin{aligned}
		\frac{1}{2} \frac{d}{dt} \norm { \widetilde{u}^{\varepsilon } (t)  }_{L^2}^2
		\leq  \left|\int_{\mathbb{R}^2}
		\nabla u :(\widetilde{u}^{\varepsilon }\otimes \widetilde{u}^{\varepsilon })(t,x) dx\right|
		+ \norm { \widetilde{g}^{\varepsilon } (t)  }_{L^2} \norm { \widetilde{u}^{\varepsilon } (t)  }_{L^2}.
	\end{aligned}
\end{equation*}
Therefore, by H\"older's inequality and the classical Biot--Savart estimate (see \cite[Section 7.1.1]{bcd11})
\begin{equation*}
 \norm {\nabla u}_{L^q}\leq C q \norm {\omega}_{L^q},
\end{equation*}
 where $q=q(t)\in (2,\infty)$ is allowed to depend on $t\in[0,T]$, we infer that 
\begin{equation*}
\begin{aligned}
\frac{1}{2} \frac{d}{dt} \norm {\widetilde{u}^{\varepsilon }(t)}_{L^2}^2  &\leq \norm {\nabla u(t)}_{L^q}\norm {\widetilde{u}^{\varepsilon }(t)}_{L^{\frac{2q}{q-1}}}^2 + \norm {\widetilde g^{\varepsilon } (t)}_{L^2} \norm {\widetilde{u}^{\varepsilon }(t)}_{L^{2}}
\\
&\leq Cq \norm {\omega (t)}_{L^2 \cap L^\infty}\norm {\widetilde{u}^{\varepsilon }(t)}_{L^{2}}^{2-\frac{2}{q}} \norm {\widetilde{u}^{\varepsilon }(t)}_{L^{\infty}}^{ \frac{2}{q}}+ \norm {\widetilde g^{\varepsilon }(t)}_{L^2} \norm {\widetilde{u}^{\varepsilon }(t)}_{L^{2}}  .
\end{aligned}
\end{equation*}
Thus, by virtue of the control  \eqref{Transport:ES} and the Gagliardo--Nirenberg interpolation inequality 
$$ \norm  {\widetilde u^\varepsilon (t) }_{ L^\infty}   \lesssim
\norm {\widetilde u^\varepsilon(t) }_{L^2} ^{\frac{1}{2}  }\norm {\widetilde \omega^\varepsilon(t) }_{L^\infty}^{\frac{1}{2} }, $$
 we find that
\begin{equation*}
	\begin{aligned}
		\frac 12\norm {\widetilde{u}^{\varepsilon }(t)}_{L^2}^2
		&\leq
		\frac 12\norm {\widetilde{u}^{\varepsilon }_0}_{L^2}^2
		+C\int_0^t q(\tau)
		\norm {\omega (\tau)}_{L^2 \cap L^\infty}
		\norm {\widetilde{u}^{\varepsilon }(\tau)}_{L^{2}}^{2-\frac{1}{q(\tau)}}
		\norm {\widetilde \omega^\varepsilon(\tau) }_{L^\infty}^{\frac{1}{q(\tau)} }
		d\tau
		\\
		&\quad + \int_0^t
		\norm {\widetilde g^{\varepsilon }(\tau)}_{L^2} \norm {\widetilde{u}^{\varepsilon }(\tau)}_{L^{2}}
		d\tau
		\\
		&\leq
		\frac 12\norm {\widetilde{u}^{\varepsilon }_0}_{L^2}^2
		+
		\int_0^t q(\tau)C_*^{1+\frac 1{q(\tau)}}
		\norm {\widetilde{u}^{\varepsilon }(\tau)}_{L^{2}}^{2-\frac{1}{q(\tau)}}
		d\tau
		\\
		&\quad + \left(\int_0^t
		\norm {\widetilde g^{\varepsilon }(\tau)}_{L^2}
		d\tau\right)^2
		+\frac 14 \norm {\widetilde{u}^{\varepsilon }}_{L^\infty([0,t);L^{2})}^2,
	\end{aligned}
\end{equation*}
where $C_*$ is defined by \eqref{C*:def}.

Now, introducing the continuous function
\begin{equation*}
	f^\varepsilon(t)\bydef C_*^{-1}\norm {\widetilde{u}^{\varepsilon }}_{L^\infty([0,t);L^{2})},
\end{equation*}
the preceding estimate can be recast as
\begin{equation*}
		\frac 14\big(f^\varepsilon(t)\big)^2
		\leq
		\frac 12\big(f^\varepsilon(0)\big)^2
		+C_*
		\int_0^t q(\tau)
		\big(f^\varepsilon(\tau)\big)^{2-\frac{1}{q(\tau)}}
		d\tau
		+ \left(C_*^{-1}\int_0^t
		\norm {\widetilde g^{\varepsilon }(\tau)}_{L^2}
		d\tau\right)^2.
\end{equation*}
Therefore, further assuming that
$$f^\varepsilon(t) \leq 1$$
on a time interval $[0,t_*]$, for some $t_*\in(0,T]$,
and setting
\begin{equation*} 
	q(t) \bydef 2 \log \left( \frac{ e}{\big(f^\varepsilon(t)\big)^2} \right),
\end{equation*}
we deduce that
\begin{equation*}
	 \big(f^\varepsilon(t)\big)^2
		\leq
	2\big(f^\varepsilon(0)\big)^2
		+C_0C_*
		\int_0^t \log \left( \frac{ e}{f^\varepsilon(\tau)^2} \right)\big( f^\varepsilon(\tau)\big)^2
		d\tau
		+ 4\left(C_*^{-1}\int_0^t
		\norm {\widetilde g^{\varepsilon }(\tau)}_{L^2}
		d\tau\right)^2,
\end{equation*}
where we employed the observation that
$$q(t)  \big(f^\varepsilon(t)\big)^{2-\frac{1}{q(t)}} \leq C_0   \log \left( \frac{ e}{\big(f^\varepsilon(t)\big)^2} \right)\big( f^\varepsilon(t)\big)^2,$$
for some $C_0\geq 1$.

 At last,  by employing Osgood's lemma   \cite[Lemma 3.4]{bcd11} (see also \cite[Lemma A.1]{HHZ}, which can be applied directly to the present case), we arrive at the desired bound
$$\frac 1{eC_*}\norm { \widetilde{u}^{\varepsilon }   }_{L^\infty ([0,t_*]; L^2)} \leq   \Big(  \frac{\sqrt 2}{eC_*}\norm { \widetilde{u}^{\varepsilon } _0  }_{L^2}  +  \frac 2{eC_*}\norm { \widetilde{g}^{\varepsilon }   }_{L^1([0,t_*];L^2)}    \Big) ^{ \exp (-  C_0C_*t_*)}.$$
In particular, assuming that
\begin{equation*}
	\frac{\sqrt 2}{eC_*}\norm { \widetilde{u}^{\varepsilon } _0  }_{L^2}  +  \frac 2{eC_*}\norm { \widetilde{g}^{\varepsilon }   }_{L^1([0,T];L^2)}
	< e^{-\exp(C_0C_*T)}\leq 1
\end{equation*}
leads to the estimate
$$\frac 1{eC_*}\norm { \widetilde{u}^{\varepsilon }   }_{L^\infty ([0,t_*]; L^2)} \leq   \Big(  \frac{\sqrt 2}{eC_*}\norm { \widetilde{u}^{\varepsilon } _0  }_{L^2}  +  \frac 2{eC_*}\norm { \widetilde{g}^{\varepsilon }   }_{L^1([0,T];L^2)}    \Big) ^{ \exp (-  C_0C_*T)}<e^{-1}.$$
Then, a classical continuation argument allows us to deduce that $t_*$, in the left-hand side above, can be chosen to be equal to $T$. 
This establishes the stability estimate \eqref{u:CV:final0}, by possibly redefining $C_*$ up to a multiplicative universal constant.

\subsubsection{Convergence of low frequencies in $L^\infty_t\dot{H}^1$}
\label{section:low:freq}

As in the proof of Lemma \ref{lemma:extrapolation}, it is readily seen that a direct use of Plancherel's theorem gives that
\begin{equation*}
	\norm { \mathds{1}_{|D|\leq \Theta_\varepsilon}(\omega^\varepsilon - \omega )  }_{L^\infty([0,T]; L^2)} \leq  \Theta_\varepsilon  \norm {u^\varepsilon -u }_{L^\infty([0,T]; L^2)} , 
\end{equation*}
for any $\Theta_\varepsilon>0$. In particular, if
\begin{equation}\label{theta:DEF0}
  \lim_{\varepsilon \rightarrow 0}  \Theta_\varepsilon  = \infty
  \quad \text{and}\quad
  \lim_{\varepsilon \rightarrow 0} \Theta_\varepsilon  \norm {u^\varepsilon -u }_{L^\infty([0,T]; L^2)}=0,
\end{equation}
we deduce that
\begin{equation}\label{CV:LF000}
\lim_{\varepsilon \rightarrow 0}  \norm { \phi(  \Theta_\varepsilon^{-1} D) (\omega^\varepsilon - \omega )  }_{L^\infty([0,T]; L^2)} =0,
\end{equation} 
for any compactly supported cutoff function $\phi\in L^\infty(\mathbb{R}^2)$.

Moreover, if $\phi$ is assumed to be supported away from the origin, then the assumptions  on the behavior of $ \Theta_\varepsilon   $ allow us to obtain that
  \begin{equation}\label{Contin.w2}
   \lim_{\varepsilon \rightarrow 0}  \norm { \phi(  \Theta_\varepsilon^{-1} D)  \omega  }_{L^\infty([0,T]; L^2)} =0.
\end{equation}
This convergence hinges  upon the  time continuity of Yudovich's solutions. Indeed, let us suppose, by contradiction, that
\begin{equation*}
	\norm { \phi(  \Theta_{\varepsilon_k}^{-1} D)  \omega  }_{L^\infty([0,T]; L^2)}\geq\delta,
\end{equation*}
for some sequence $\varepsilon_k\to 0$, as $k\to\infty$, and some constant $\delta>0$. Then, by continuity, writing
$$\norm { \phi(  \Theta_{\varepsilon_k}^{-1} D)  \omega  }_{L^\infty([0,T]; L^2)} =\norm {\phi(  \Theta_{\varepsilon_k}^{-1} D)  \omega(t_k)}_{L^2},$$
for some suitable $t_k \in [0,T]$, and assuming, by compactness, that $t_k\to t$, we see that the control
$$\norm {\phi(  \Theta_{\varepsilon_k}^{-1} D)  \omega(t_k)}_{L^2}  \leq \norm{\phi}_{L^\infty}  \norm  { \omega(t_k) -\omega(t )}_{L^2}
+ \norm {\phi(  \Theta_{\varepsilon_k}^{-1} D)  \omega(t )}_{L^2}$$
implies that
\begin{equation*}
	\norm { \phi(  \Theta_{\varepsilon_k}^{-1} D)  \omega  }_{L^\infty([0,T]; L^2)}\to 0,
\end{equation*}
which is impossible. It follows that \eqref{Contin.w2} holds true.

We will make use of the preceding convergence properties with the particular choice of cutoff function $\phi(D)=\mathds{1}_{M \leq |D|\leq N }$, for any $0<M<N$, which yields
\begin{equation*}
 \lim_{\varepsilon\rightarrow 0} \norm { \mathds{1}_{M  \Theta_\varepsilon \leq |D|\leq N \Theta_\varepsilon  }\omega    }_{L^\infty([0,T]; L^2)}=0, 
\end{equation*} 
and
\begin{equation}\label{compactness1}
\begin{aligned}
 \lim_{\varepsilon\rightarrow 0} \norm { \mathds{1}_{M  \Theta_\varepsilon \leq |D|\leq N \Theta_\varepsilon }\omega^\varepsilon   }_{L^\infty([0,T]; L^2)}   = 0 ,
\end{aligned}
\end{equation} 
for any choice of parameters $\Theta_\varepsilon$ satisfying \eqref{theta:DEF0}.

\subsubsection{Convergence of high frequencies in $L^\infty_t\dot{H}^1$}
\label{section:high:freq}

We first introduce some notation. Let 
$$\mathcal{C} \bydef   \left \{\xi \in \mathbb{R}^2 :  \frac{3}{4}\leq |\xi | \leq \frac{8}{3}\right \} , \qquad \mathcal{B} \bydef   \left \{\xi \in \mathbb{R}^2 :  |\xi | \leq \frac{4}{3}\right \}, $$ 
and   consider    smooth radial functions $\varphi \in \mathcal{  D}( \mathcal{C})$ and $ \psi \in \mathcal{  D}( \mathcal{B})  $ with
$$0 \leq  \varphi (\xi),  \psi (\xi) \leq 1, \quad  \text{for all }  \xi \in \mathbb{R}^2,$$
and
$$\psi(\xi)=1, \quad  \text{for all }|\xi|\leq 1.$$
Furthermore, it is possible to ensure that the family of functions
$$ \varphi_j(\cdot) \bydef \varphi  (2^{-j} \cdot )\in \mathcal{  D}(2^j\mathcal{C}), \quad  j\in \mathbb{Z},$$
 provides us with partitions of unity
$$1 = \sum_{ j\in \mathbb{Z}} \varphi_j(\xi) , \quad \text{for all } \xi\in \mathbb{R}^2\setminus \{0\}$$
and
$$1 =   \sum_{ j\geq 0} \varphi_j(\xi) + \psi(\xi), \quad \text{for all } \xi\in \mathbb{R}^2.$$
In this case, the corresponding convolution operators  
\begin{equation}\label{S0:def}
S_0^\varepsilon   \bydef \psi(\Theta_\varepsilon^{-1}D)    ,\qquad \Delta_j^\varepsilon \bydef \varphi_j( \Theta_\varepsilon^{-1} D),
\end{equation}
where $\Theta_\varepsilon>0$ is any parameter satisfying \eqref{theta:DEF0}, satisfy that
$$ S_0^\varepsilon+  \sum_{j\geq 0}\Delta_j^\varepsilon= \mathrm{Id}.$$

Next, further introducing the Fourier multiplier operators
$$ \sqrt{S_0^\varepsilon} \bydef \sqrt{\psi   (\Theta_\varepsilon^{-1} D)}
\qquad\text{and}\qquad
\sqrt{\mathrm{Id}-S_0^\varepsilon}\bydef \sqrt{1-\psi  (\Theta_\varepsilon^{-1}D)} ,$$
and observing that
\begin{equation*}
	\left( \sum_{j\geq 0} \varphi_j^\varepsilon(\xi) \right)^2
	\leq \sum_{j\geq 0} \varphi_j^\varepsilon(\xi),
\end{equation*}
it holds, for any given $f\in L^2$, that
\begin{equation}\label{M1}
		\norm{ (\mathrm{Id}-S_0^\varepsilon )f  }_{L^2}
		\leq
		\norm{ \sqrt{\mathrm{Id}-S_0^\varepsilon }f}_{L^2},
\end{equation}
which will come in handy later on.

Now, observe that \eqref{Contin.w2} yields 
$$ \lim_{\varepsilon\rightarrow 0}\norm { (\mathrm{Id} -S_0^\varepsilon) \omega   }_{L^\infty( [0,T]; L^2)}=0 .  $$
Therefore, by \eqref{CV:LF000}, we conclude that the convergence
$$ \lim_{\varepsilon\rightarrow 0} \norm {  \omega^ \varepsilon- \omega }_{L^\infty( [0,T]; L^2)}=0 $$
is equivalent to
\begin{equation}\label{vanishing:1}
	\lim_{\varepsilon\rightarrow 0}\norm { (\mathrm{Id} -S_0^\varepsilon) \omega^\varepsilon   }_{L^\infty( [0,T]; L^2)}=0 ,
\end{equation}
which can be interpreted as a time-dependent version of the compactness criterion given in Lemma \ref{lemma:extrapolation}.

Thus, in order to prove Theorem \ref{thm2}, there only remains to establish \eqref{vanishing:1}. To that end, we first recall that $\omega^\varepsilon$ solves the transport equation
\begin{equation}\label{omega-equa}
 \partial_t \omega^\varepsilon + u^\varepsilon \cdot \nabla \omega^\varepsilon = \curl g^\varepsilon.
\end{equation}
Then, formally taking the  inner product of  \eqref{omega-equa}   with $  \sum_{j\geq 0}\Delta_j^\varepsilon \omega^\varepsilon $ and using the divergence-free condition of $u^\varepsilon$, we find that
\begin{equation}\label{energy-HF}
\begin{aligned}
 \frac{1}{2 }\norm{  \sqrt{\mathrm{Id} - S_0^\varepsilon} \omega^\varepsilon (t )   }_{L^2}^2 &= \frac{1}{2 }\norm{  \sqrt{\mathrm{Id} - S_0^\varepsilon} \omega^\varepsilon_0 }_{L^2}^2+ \int_0^t  \int_{\mathbb{R}^2}  \curl g^\varepsilon  (\mathrm{Id} - S_0^\varepsilon) \omega^\varepsilon(\tau,x)  dxd\tau\\
& \quad -\underbrace{ \int_0^t  \int_{\mathbb{R}^2}u^\varepsilon\cdot \nabla \left( \sum_{j\leq -1}\Delta_j^\varepsilon \omega^\varepsilon  \right)  \sum_{j\geq 0}\Delta_j^\varepsilon\omega^\varepsilon(\tau,x)     dxd\tau }_{\bydef\mathcal{J}(t)} .
\end{aligned}
\end{equation}
For the sake of completeness, we provide a rigorous justification of \eqref{energy-HF} in Section \ref{section:energy:identity}, below.

We show now how to estimate $\mathcal{J}(t)$. To that end, we split it into a sum of three terms
\begin{equation}\label{J:decomposition}
	\begin{aligned}
		\mathcal{J}_1(t) &\bydef \int_0^t \int_{\mathbb{R}^2}      u^\varepsilon  \cdot \nabla \left(\Delta_{-1} ^\varepsilon \omega^\varepsilon(\tau)  \right) \Delta_0^\varepsilon\omega^\varepsilon(\tau,x)   dx d\tau ,
		\\
		\mathcal{J}_2(t) &\bydef   \int_0^t  \int_{\mathbb{R}^2}     u^\varepsilon   \cdot \nabla \left(  \sum_{j\leq -2}\Delta_j^\varepsilon \omega^\varepsilon \right) \sum_{j\geq 0}\Delta_j^\varepsilon\omega^\varepsilon(\tau,x)     dxd\tau ,
		\\
		\mathcal{J}_3(t) &\bydef   \int_0^t  \int_{\mathbb{R}^2}     u^\varepsilon  \cdot \nabla \left(  \Delta_{-1}^\varepsilon \omega^\varepsilon  \right) \sum_{j\geq 1}\Delta_j^\varepsilon\omega^\varepsilon(\tau,x)   dxd\tau ,
	\end{aligned}
\end{equation}
which we control separately.

The term $\mathcal{J}_1 $ is the most difficult to estimate. Specifically, one can show that
\begin{equation}\label{estimate:J1}
	|\mathcal{J}_1(t)| \lesssim \sum_{i=-1}^0 \int_0^t\norm {\nabla u^\varepsilon(\tau)}_{L^2}  \norm {\Delta_i^\varepsilon  \omega^\varepsilon(\tau)}_{L^4}^2 d\tau.
\end{equation}
For the sake of clarity, we defer the justification of this estimate to Section \ref{section:estimate:J1}, below.

Then, observe that
\begin{equation*}
	|\mathcal{J}_1(t)|
	\lesssim  \norm {\omega^\varepsilon }_{L^\infty([0,t]; L^2 )}
	\norm {\omega^\varepsilon }_{L^\infty([0,t]; L^\infty )}
	\sum_{i=-1}^0  \norm {\Delta_i^\varepsilon  \omega^\varepsilon  } _{L^1([0,t];L^2)},
\end{equation*}
which follows from H\"older's inequality and the Biot--Savart estimate.

As for $  \mathcal{J}_2$ and $\mathcal{J}_3$, we begin with exploiting the support localizations
\begin{equation*}
	\begin{aligned}
		\supp \mathcal{F}\left(  \sum_{j\leq -2}\Delta_j^\varepsilon \omega^\varepsilon  \right)
		&= \left\{ \xi\in \mathbb{R}^2: |\xi| \leq \frac{2\Theta_\varepsilon}{3} \right\}, 
		\\
		\supp \mathcal{F}\left(  \sum_{j\geq 0}\Delta_j^\varepsilon\omega^\varepsilon   \right)
		&= \left\{ \xi\in \mathbb{R}^2: |\xi| \geq  \frac{3\Theta_\varepsilon}{4} \right\}  , 
	\end{aligned}
\end{equation*}
to deduce that  
 $$ \supp \mathcal{F}\left(  \sum_{j\leq -2}\Delta_j^\varepsilon \omega^\varepsilon  \sum_{j\geq 0}\Delta_j^\varepsilon\omega^\varepsilon   \right) \subset  \left\{ \xi\in \mathbb{R}^2: |\xi| \geq  \frac{\Theta_\varepsilon}{12} \right\}.$$
Similarly,  the fact that
\begin{equation*}
	\begin{aligned}
		\supp \mathcal{F}\left(   \Delta_{-1}^\varepsilon \omega^\varepsilon  \right) &\subset \left\{ \xi\in \mathbb{R}^2: |\xi| \leq \frac{4\Theta_\varepsilon}{3} \right\},
		\\
		\supp \mathcal{F}\left(  \sum_{j\geq 1}\Delta_j^\varepsilon\omega^\varepsilon   \right)&\subset \left\{ \xi\in \mathbb{R}^2: |\xi| \geq  \frac{3\Theta_\varepsilon}{2} \right\}  , 
	\end{aligned}
\end{equation*}
entails that
 $$  \supp\mathcal{F}\left(  \Delta_{-1}^\varepsilon \omega^\varepsilon  \sum_{j\geq 1}\Delta_j^\varepsilon\omega^\varepsilon   \right) \subset  \left\{ \xi\in \mathbb{R}^2: |\xi| \geq  \frac{ \Theta_\varepsilon}{6} \right\}.$$
Accordingly, we infer that 
 $$\begin{aligned}
 \mathcal{J}_2(t) 
 &=   \int_0^t  \int_{\mathbb{R}^2}    \left( \mathds{1}_{|D|\geq  \frac{1}{12}\Theta_\varepsilon}u^\varepsilon \right)  \cdot \nabla \left(  \sum_{j\leq -2}\Delta_j^\varepsilon \omega^\varepsilon   \right) \sum_{j\geq 0}\Delta_j^\varepsilon\omega^\varepsilon(\tau,x)  dxd\tau ,
 \end{aligned}$$
 and 
  $$\begin{aligned}
 \mathcal{J}_3(t) 
 &=   \int_0^t  \int_{\mathbb{R}^2}    \left( \mathds{1}_{|D|\geq  \frac{1}{6}\Theta_\varepsilon} u^\varepsilon \right)  \cdot \nabla \left(  \Delta_{-1}^\varepsilon \omega^\varepsilon  \right) \sum_{j\geq 1}\Delta_j^\varepsilon\omega^\varepsilon(\tau,x)     dxd\tau.
 \end{aligned}$$

 Consequently,  since $u^\varepsilon$ above is localized in its high frequencies, it holds that
\begin{equation}\label{J2:ES1}
\begin{aligned}
 |\mathcal{J}_2(t)| &\lesssim   \int_0^t\norm { \mathds{1}_{|D|\geq \frac{1}{12}\Theta_\varepsilon}\nabla u^\varepsilon(\tau) }_{L^2}  \norm {  ( S_0^\varepsilon -\Delta_{-1}^\varepsilon)\omega^\varepsilon(\tau)}_{L^\infty}  \norm { (\mathrm{Id}-S_0^\varepsilon  )\omega^\varepsilon(\tau)}_{L^2}  d\tau\\
  &\lesssim  \int_0^t\norm {\mathds{1}_{|D|\geq  \frac{1}{12}\Theta_\varepsilon}\omega^\varepsilon(\tau)  }_{L^2}  \norm {   \omega^\varepsilon(\tau)}_{L^\infty}  \norm { (\mathrm{Id}-S_0^\varepsilon  )\omega^\varepsilon(\tau)}_{L^2}  d\tau ,
 \end{aligned}
\end{equation}  
 and, in a similar fashion, that
\begin{equation}\label{J2:ES2}
\begin{aligned}
 |\mathcal{J}_3(t)| &\lesssim   \int_0^t\norm {\mathds{1}_{|D|\geq \frac{1}{6}\Theta_\varepsilon}\nabla u^\varepsilon(\tau)  }_{L^2}  \norm {    \Delta_{-1}^\varepsilon   \omega^\varepsilon(\tau)}_{L^\infty}  \norm { (\mathrm{Id}-S_0^\varepsilon -\Delta_0^\varepsilon)\omega^\varepsilon(\tau)}_{L^2}  d\tau\\
  &\lesssim  \int_0^t\norm { \mathds{1}_{|D|\geq  \frac{1}{12}\Theta_\varepsilon}\omega^\varepsilon(\tau) }_{L^2}  \norm {   \omega^\varepsilon(\tau)}_{L^\infty} \Big( \norm {   \Delta_0^\varepsilon \omega^\varepsilon(\tau)}_{L^2}  + \norm { (\mathrm{Id}-S_0^\varepsilon )\omega^\varepsilon(\tau)}_{L^2} \Big) d\tau .
 \end{aligned}
\end{equation}  
 Therefore, the bounds
 \begin{equation*}
 \norm {   \Delta_0^\varepsilon \omega^\varepsilon(\tau)}_{L^2}  \leq   \norm { \mathds{1}_{\frac{1}{12}\Theta_\varepsilon \leq |D|\leq \frac{8	}{3} \Theta_\varepsilon}\omega^\varepsilon(\tau) }_{L^2}
\end{equation*}     
and  
  $$\begin{aligned}
  \norm { \mathds{1}_{|D|\geq \frac{1}{12}\Theta_\varepsilon}\omega^\varepsilon(\tau) }_{L^2} ^2  &=  \norm { \mathds{1}_{\frac{1}{12} \Theta_\varepsilon \leq |D|< \frac{4	}{3} \Theta_\varepsilon}\omega^\varepsilon(\tau) }_{L^2}^2 + \norm {\mathds{1}_{|D|\geq \frac{4}{3}\Theta_\varepsilon}\omega^\varepsilon(\tau)  }_{L^2}^2  \\
  &\leq   \norm {\mathds{1}_{\frac{1}{12} \Theta_\varepsilon\leq |D|\leq \frac{8	}{3} \Theta_\varepsilon}\omega^\varepsilon(\tau)  }_{L^2}^2+  \norm { (\mathrm{Id}- S_0^\varepsilon)\omega^\varepsilon(\tau)  }_{L^2} ^2,
 \end{aligned}$$ 
lead to
  $$\begin{aligned}
 |\mathcal{J}_2(t)| +  |\mathcal{J}_3(t)|
  &\lesssim    \norm {\omega^\varepsilon }_{L^\infty ([0,t];L^2\cap L^\infty )}^2  \norm {  \mathds{1}_{\frac{1}{12}\Theta_\varepsilon \leq |D|\leq \frac{8	}{3}  \Theta_\varepsilon}\omega^\varepsilon  }_{L^1 ([0,t];L^2)}
  \\
  &\quad+ \int _0^t \norm {\omega^\varepsilon(\tau)}_{L^\infty} \norm { (\mathrm{Id}- S_0^\varepsilon)\omega^\varepsilon(\tau)  }_{L^2} ^2 d\tau.
 \end{aligned}$$

Finally, gathering all estimates on $\mathcal{J}_1$, $\mathcal{J}_2$ and $\mathcal{J}_3$, employing the control 
$$ \norm {\omega^\varepsilon}_{L^\infty ([0,T]; L^2\cap L^\infty)} \leq \sup_{\varepsilon>0} \left(\norm {\omega^\varepsilon_0}_{L^2\cap L^\infty} + \norm {\curl g ^\varepsilon}_{L^1([0,T]; L^2\cap L^\infty)} \right)\lesssim  C_*,$$
where $C_*$ is defined in \eqref{C*:def}, and using \eqref{M1},
we arrive at the bound
\begin{equation*}
	|\mathcal{J}(t)|
	\lesssim
	C_*^2 \norm {  \mathds{1}_{\frac{1}{12}\Theta_\varepsilon \leq |D|\leq \frac{8	}{3} \Theta_\varepsilon}\omega^\varepsilon  }_{L^1 ([0,T];L^2)} +C_* \int _0^t   \norm { \sqrt{\mathrm{Id}-S_0^\varepsilon }\omega^\varepsilon(\tau)  }_{L^2} ^2 d\tau,
\end{equation*}
for any $t\in [0,T]$.

Further incorporating this estimate into \eqref{energy-HF}, we deduce that
\begin{equation*} 
	\begin{aligned}
		\norm{ \sqrt{\mathrm{Id}-S_0^\varepsilon } \omega^\varepsilon (t )   }_{L^2}^2
		&\leq
		\norm{ \sqrt{\mathrm{Id}-S_0^\varepsilon } \omega^\varepsilon_0 }_{L^2}^2+ C_*\norm  { (\mathrm{Id} - S_0^\varepsilon) \curl g ^\varepsilon}_{L^1 ([0,T];L^2)}
		\\
		& \quad +C_*^2 \norm {  \mathds{1}_{\frac{1}{12}\Theta_\varepsilon \leq |D|
		\leq \frac{8	}{3} \Theta_\varepsilon}\omega^\varepsilon  }_{L^1 ([0,T];L^2)}
		+C_* \int _0^t
		\norm { \sqrt{\mathrm{Id}-S_0^\varepsilon }\omega^\varepsilon(\tau)  }_{L^2} ^2 d\tau.
	\end{aligned}
\end{equation*}
At last, applying Gr\"onwall's lemma  
yields that
\begin{equation*}
\begin{aligned}
\norm{ \sqrt{\mathrm{Id}-S_0^\varepsilon } \omega^\varepsilon   }_{L^\infty([0,T];  L^2)}^2
&\lesssim_{C_*}
\bigg(\norm{ \sqrt{\mathrm{Id}-S_0^\varepsilon } \omega^\varepsilon_0 }_{L^2}^2
+  \norm { \mathds{1}_{\frac{1}{12}\Theta_\varepsilon \leq |D|\leq \frac{8	}{3} \Theta_\varepsilon}\omega^\varepsilon  }_{L^1 ([0,T]; L^2)}
\\
&\qquad+   \norm  { (\mathrm{Id} - S_0^\varepsilon) \curl
g^\varepsilon}_{L^1 ([0,T]; L^2)} \bigg)
\exp\left( C_*  T\right).
\end{aligned}
\end{equation*}
Clearly, in view of \eqref{compactness1}, the right-hand side above vanishes in the limit $\varepsilon\rightarrow 0$. We therefore conclude that \eqref{vanishing:1} holds true, thereby completing the proof of Theorem \ref{thm2}. \qed

\subsubsection{Justification of \eqref{energy-HF}}\label{section:energy:identity}

Note that, in order to establish \eqref{energy-HF}, we have taken advantage of the formal cancellation
 $$ \int_{\mathbb{R}^2} u^\varepsilon(\tau,x)\cdot \nabla \left(\sum_{j\geq 0}\Delta_j^\varepsilon \omega^\varepsilon(\tau,x) \right) \left(\sum_{j\geq 0}\Delta_j^\varepsilon \omega^\varepsilon(\tau,x) \right)dx =0, \quad \text{for all } \tau \in [0,T],$$
despite the fact that the   integral above is not well-defined, for $\omega^\varepsilon(t)$   only belongs to Lebesgue spaces. Here, we show  that \eqref{energy-HF} can be justified without relying on the above identity.

To that end, we first observe that
$$ \int_{\mathbb{R}^2} \div(u^\varepsilon  \omega^\varepsilon ) S_0^\varepsilon\omega^ \varepsilon (\tau,x) dx = -\int_{\mathbb{R}^2} u^\varepsilon\cdot \nabla (S_0^\varepsilon\omega^\varepsilon )(\mathrm{Id}-S_0^\varepsilon) \omega^\varepsilon (\tau,x) dx,$$
for any $\tau \in [0,T]$. Accordingly, by taking the inner product of the transport equation \eqref{omega-equa} with $S_0^\varepsilon \omega^\varepsilon$, we obtain, for any $t\in [0,T]$, that
\begin{equation*} 
\begin{aligned}
\frac 12 \norm {  \sqrt{S_0^\varepsilon}  \omega^\varepsilon(t)  }_{L^2}^2  
& = \frac 12 \norm {\sqrt{S_0^\varepsilon}\omega^\varepsilon_0  }_{L^2}^2  + \int_0^t\int_{\mathbb{R}^2}   \curl g^\varepsilon     S_0^\varepsilon\omega^\varepsilon (\tau,x)   dxd\tau \\
& \quad +\int_0^t  \int_{\mathbb{R}^2} u^\varepsilon\cdot \nabla (S_0^\varepsilon\omega^\varepsilon )(\mathrm{Id}-S_0^\varepsilon) \omega^\varepsilon (\tau,x)  dxd\tau  .
\end{aligned}
\end{equation*} 
 On the other hand, we know, for any $t\in [0,T]$, that
 \begin{equation*} 
\begin{aligned}
\frac 12\norm {   \omega^\varepsilon(t)  }_{L^2}^2  
& =  \frac 12 \norm {\omega^\varepsilon_0  }_{L^2}^2  + \int_0^t\int_{\mathbb{R}^2}  \curl g^\varepsilon     \omega^\varepsilon  (\tau,x)    dxd\tau .
\end{aligned}
\end{equation*}
 Consequently, we see that \eqref{energy-HF} follows from the combination of the two preceding identities with
$$\norm{ \omega^\varepsilon(t)}_{L^2}^2 = \norm{\sqrt{S_0^\varepsilon} \omega^\varepsilon(t) }_{L^2}^2 + \norm{ \sqrt{(\mathrm{Id}-S_0^\varepsilon)}\omega^\varepsilon(t)  }_{L^2}^2,$$
which completes its justification.\qed

\subsubsection{Justification of \eqref{estimate:J1}}\label{section:estimate:J1}

Here, we give a complete proof of estimate \eqref{estimate:J1} on $\mathcal{J}_1$, which follows directly from an application of Lemma \ref{T3-ES}, below. We proceed in several steps. First, in Lemma \ref{standard:lemma}, we summarize the fundamental properties of the partition of unity introduced in the preceding steps. Then, in Lemma \ref{lemma:iden:1}, we establish a crucial identity which is a natural consequence of the localization of frequencies in the nonlinear advection term $u\cdot \nabla \omega$. Finally, we combine the results of Lemmas \ref{standard:lemma} and \ref{lemma:iden:1} to deduce Lemma \ref{T3-ES}.

\begin{lem}\label{standard:lemma}
	The dyadic frequency decomposition operators given in \eqref{S0:def} satisfy the identities
	\begin{gather}
		\nonumber
		\Delta_i^\varepsilon\Delta_{j}^\varepsilon=0,
		\\
		\label{identity:2}
		\Delta_i ^\varepsilon=   \Delta_i^\varepsilon   \sum_{k  \in \{i,i\pm 1 \}} \Delta_{k}^\varepsilon ,
		\\
		\label{identity:3}
		\left( \Delta_i^\varepsilon  +  \Delta_{i+1}^\varepsilon \right)  \Delta_i^\varepsilon \Delta_{i+1}^\varepsilon  =  \Delta_i ^\varepsilon \Delta_{i+1}^\varepsilon,
	\end{gather}
	for any $i,j\in\mathbb{Z}$, with $|i-j|\geq 2$.
	
	Furthermore, we have that
	\begin{equation}\label{support:1}
		\supp { \mathcal{F}\left( \Delta_i^\varepsilon f
		\Delta_{i+1}^\varepsilon
		\Delta_{i+2}^\varepsilon g \right)}
		\subset \left \{  |\xi| \geq \frac{\Theta_\varepsilon 2^i }{3}   \right\} ,
	\end{equation}
	for any tempered distributions $f$ and $g$.
\end{lem}

\begin{proof}
	This is a consequence of the localization of the support of the function $\varphi_j $, which defines the dyadic block $\Delta_j^\varepsilon$. We refer to \cite[Proposition 2.10]{bcd11} for the proof of the first three identities.
	
	As for \eqref{support:1}, it follows from the fact that
 $$ \supp \left ( \varphi_{i+1}^\varepsilon   \varphi_{i+2}^\varepsilon\right) \subset  \left \{\xi \in \mathbb{R}^2 :  3\cdot 2^i\leq \frac{|\xi |}{\Theta_\varepsilon} \leq \frac{8}{3} \cdot 2^{i+1}\right \} ,$$
 which implies that
 $$ \supp \varphi_i^\varepsilon +  \supp \left ( \varphi_{i+1}^\varepsilon   \varphi_{i+2}^\varepsilon\right) \subset  \left \{\xi \in \mathbb{R}^2 :   \frac{|\xi |}{\Theta_\varepsilon}\geq  \frac{ 2^{i}}{3}   \right \},$$ 
	thereby completing the proof.
\end{proof}

\begin{lem}\label{lemma:iden:1}
	Let $u$ be a divergence-free vector field in $   L^2(\mathbb{R}^2) $ and $\omega$  be a real-valued function in $ L^2(\mathbb{R}^2)$.
	Then, for any $j\in \mathbb{Z}$,   it holds that 
 $$ \begin{aligned}
 \int_{\mathbb{R}^2} u  \cdot \nabla \Delta_{j}^\varepsilon \omega \Delta_{j+1}^\varepsilon \omega  dx   =     \mathcal{I}_1 + \mathcal{I}_2,
\end{aligned} 
 $$
 where 
  $$ \begin{aligned}
  \mathcal{I}_1  &\bydef \int_{\mathbb{R}^2} \left[ \Delta_{j+1}^\varepsilon, u \cdot \nabla  \right] \Delta_{j}^\varepsilon \omega      \left( \Delta_j^\varepsilon +  \Delta_{j+1}^\varepsilon \right)   \omega   dx  \\
   & \quad-\frac{1}{2}  \int_{\mathbb{R}^2} \left[  \Delta_{j}^\varepsilon\Delta_{j+1}^\varepsilon , u \cdot \nabla \right]  \left( \Delta_{j}^\varepsilon +  \Delta_{j+1}^\varepsilon   \right)  \omega    \left( \Delta_{j}^\varepsilon +  \Delta_{j+1}^\varepsilon   \right)  \omega dx ,
  \end{aligned}$$
 and 
 $$\mathcal{I}_2 \bydef  \int_{\mathbb{R}^2} \mathds{1}_{|D|> \frac{\Theta_\varepsilon 2^j}{3}} u   \cdot \nabla \Delta_{j}^\varepsilon \omega         \Delta_{j+1}^\varepsilon\Delta_{j+2}^\varepsilon  \omega dx .$$
 \end{lem}

 \begin{proof}
	For the sake of simplicity, we assume that  $j=0$. The general case $j\in\mathbb{Z}$ follows from a similar argument.
	
	We begin by utilizing identity \eqref{identity:2} to write that
 $$ \begin{aligned}
 \int_{\mathbb{R}^2} u  \cdot \nabla \Delta_{0}^\varepsilon \omega  \Delta_{1}^\varepsilon \omega  =   
 \int _{\mathbb{R}^2} u  \cdot \nabla \Delta_{0}^\varepsilon \omega     \left( \Delta_0^\varepsilon +  \Delta_{1}^\varepsilon   \right) \Delta_{1}^\varepsilon \omega   +  \int_{\mathbb{R}^2} u  \cdot \nabla \Delta_{0}^\varepsilon\omega         \Delta_{1}^\varepsilon  \Delta_{2}^\varepsilon\omega  \bydef \mathcal{K}_1 + \mathcal{K}_2.
\end{aligned} 
 $$
	Then, by virtue of \eqref{support:1}, the expression $\mathcal{K}_2$  can be recast as
 $$\mathcal{K}_2 =  \int_{\mathbb{R}^2}  \mathds{1}_{|D|> \frac{\Theta_\varepsilon}{3}}u   \cdot \nabla \Delta_{0}^\varepsilon \omega        \Delta_{1}^\varepsilon  \Delta_{2}^\varepsilon \omega ,$$
which is precisely $\mathcal{I}_2$.

 As for $\mathcal{K}_1$, we split it into two parts  
 $$\mathcal{K}_1 = \int_{\mathbb{R}^2} \left[ \Delta_1^\varepsilon,u \cdot \nabla \right] \Delta_{0}^\varepsilon \omega      \left( \Delta_0^\varepsilon +  \Delta_{1}^\varepsilon   \right)   \omega + \int_{\mathbb{R}^2} u  \cdot \nabla \Delta_{0}^\varepsilon\Delta_{1}^\varepsilon \omega     \left( \Delta_0^\varepsilon +  \Delta_{1}^\varepsilon   \right)  \omega  \bydef \mathcal{K}_{11} + \mathcal{K}_{12},$$
where we utilized the fact that $\Delta_1^\varepsilon $ is a self-adjoint operator combined with the identity 
 \begin{equation*}
 	 \Delta_1^\varepsilon \left( u  \cdot \nabla \Delta_{0}^\varepsilon \omega\right)=   \left[ \Delta_1^\varepsilon,u \cdot \nabla \right] \Delta_{0}^\varepsilon \omega +  u  \cdot \nabla \Delta_{0}^\varepsilon\Delta_{1}^\varepsilon \omega ,
 \end{equation*}
which follows from the definition of the commutator symbol.
In particular, observe that $\mathcal{K}_{11}$ already matches the first term in $\mathcal{I}_1$.
 
 Regarding $\mathcal{K}_{12}$, we perform an integration by parts followed by \eqref{identity:3} to deduce that
 $$ \begin{aligned}
 \mathcal{K}_{12} & =- \int_{\mathbb{R}^2} u  \cdot \nabla  \left( \Delta_0^\varepsilon +  \Delta_{1}^\varepsilon   \right)  \omega    \Delta_{0}^\varepsilon\Delta_{1}^\varepsilon \omega   \\
   & = -  \int _{\mathbb{R}^2} u  \cdot \nabla  \left( \Delta_0^\varepsilon +  \Delta_{1}^\varepsilon   \right)  \omega  \Delta_{0}^\varepsilon\Delta_{1}^\varepsilon \left( \Delta_0^\varepsilon +  \Delta_{1}^\varepsilon   \right)  \omega,
 \end{aligned}$$
which, employing commutators, can be rewritten as
   $$ \begin{aligned}
   \mathcal{K}_{12} &  = -  \int_{\mathbb{R}^2} \left[  \Delta_{0}^\varepsilon\Delta_{1}^\varepsilon, u \cdot \nabla \right]  \left( \Delta_0^\varepsilon +  \Delta_{1}^\varepsilon   \right)  \omega   \left( \Delta_0^\varepsilon +  \Delta_{1}^\varepsilon   \right)  \omega \\
   &\quad-    \int_{\mathbb{R}^2} u  \cdot \nabla      \Delta_{0}^\varepsilon\Delta_{1}^\varepsilon\left( \Delta_0^\varepsilon +  \Delta_{1}^\varepsilon   \right)  \omega     \left( \Delta_0^\varepsilon +  \Delta_{1}^\varepsilon   \right)  \omega .
 \end{aligned}
 $$
 Then, we use \eqref{identity:3}, again, to find that 
   $$ \begin{aligned}
   \mathcal{K}_{12} & = -  \int_{\mathbb{R}^2} \left[  \Delta_{0}^\varepsilon\Delta_{1}^\varepsilon, u \cdot \nabla \right]  \left( \Delta_0^\varepsilon +  \Delta_{1}^\varepsilon   \right)  \omega   \left( \Delta_0^\varepsilon +  \Delta_{1}^\varepsilon   \right)  \omega  -    \int_{\mathbb{R}^2} u  \cdot \nabla     \Delta_{0}^ \varepsilon\Delta_{1}^\varepsilon  \omega    \left( \Delta_0^\varepsilon +  \Delta_{1}^\varepsilon   \right)  \omega \\
   & = -  \int_{\mathbb{R}^2} \left[ \Delta_{0}^\varepsilon\Delta_{1}^\varepsilon,u \cdot \nabla \right]  \left( \Delta_0^\varepsilon +  \Delta_{1}^\varepsilon   \right)  \omega  \left( \Delta_0^\varepsilon +  \Delta_{1}^\varepsilon   \right)  \omega  - \mathcal{K}_{12}.
 \end{aligned}$$
Consequently, we arrive at the conclusion that 
 $$ \mathcal{K}_{12} = -\frac{1}{2}  \int_{\mathbb{R}^2} \left[   \Delta_{0}^\varepsilon\Delta_{1}^\varepsilon , u \cdot \nabla\right]  \left( \Delta_0^\varepsilon +  \Delta_{1}^\varepsilon   \right)  \omega    \left( \Delta_0^\varepsilon +  \Delta_{1}^\varepsilon   \right)  \omega ,  $$
 thereby completing the proof of the lemma by successfully identifying all the terms in $\mathcal{I}_1$. 
 \end{proof}

\begin{lem}\label{T3-ES}
	For any $u\in \dot{H}^1(\mathbb{R}^2)$ and $\omega \in L^4(\mathbb{R}^2)$,
	it holds that
	$$ \begin{aligned}
	\left| \int_{\mathbb{R}^2} u \cdot \nabla \Delta_{j}^\varepsilon \omega \Delta_{j+1}^\varepsilon \omega  \right| \lesssim    \norm {\nabla u  }_{L^2} \left( \norm {\Delta_j^\varepsilon  \omega  }_{L^4}^2 + \norm {\Delta_{j+1}^\varepsilon  \omega  }_{L^4}^2\right)  ,
	\end{aligned} 
	$$
	for all $j\in \mathbb{Z}$.
	
	More precisely, there is a decomposition
	$$ \int_{\mathbb{R}^2} u \cdot \nabla \Delta_{j}^\varepsilon \omega \Delta_{j+1}^\varepsilon \omega =  \mathcal{I}_{1} + \mathcal{I}_2,$$
	where 
	$$ |\mathcal{I}_1| \lesssim  \norm {\nabla u  }_{L^ q}
	\left(
	\norm { \Delta_j^\varepsilon \omega     }_{L^\frac{2q}{q-1} }^2
	+
	\norm {\Delta_{j+1}^\varepsilon \omega     }_{L^\frac{2q}{q-1} }^2
	\right)$$
	and 
	$$ |\mathcal{I}_2   | \lesssim \norm {\mathds{1}_{|D|>\frac{\Theta_\varepsilon 2^j}{3}}  \nabla u}_{L^2} \norm { \Delta_j^\varepsilon \omega}_{L^{2p}} \norm{\Delta_{j+1}^\varepsilon\omega}_{L^{\frac{2p}{p-1}}},$$
	for any $p,q\in [1,\infty]$ and $j\in\mathbb{Z}$.
\end{lem}

 \begin{proof}
	The decomposition into $\mathcal{I}_1+\mathcal{I}_2$ follow from an application of Lemma \ref{lemma:iden:1}. Then, a direct application of H\"older's inequality gives that
	\begin{equation*}
		\begin{aligned}
			|\mathcal{I}_2   |
			&\lesssim
			\norm {\mathds{1}_{|D|>\frac{\Theta_\varepsilon 2^j}{3}} u}_{L^2} \norm {\nabla \Delta_j^\varepsilon \omega}_{L^{2p}} \norm{\Delta_{j+1}^\varepsilon\omega}_{L^{\frac{2p}{p-1}}}
			\\
			&\lesssim 2^j\Theta_\varepsilon
			\norm {\mathds{1}_{|D|>\frac{\Theta_\varepsilon 2^j}{3}} u}_{L^2} \norm { \Delta_j^\varepsilon \omega}_{L^{2p}} \norm{\Delta_{j+1}^\varepsilon\omega}_{L^{\frac{2p}{p-1}}}
			\\
			&\lesssim 
			\norm {\mathds{1}_{|D|>\frac{\Theta_\varepsilon 2^j}{3}} \nabla u}_{L^2} \norm { \Delta_j^\varepsilon \omega}_{L^{2p}} \norm{\Delta_{j+1}^\varepsilon\omega}_{L^{\frac{2p}{p-1}}},
		\end{aligned}
	\end{equation*}
	which is the desired bound on $\mathcal{I}_2$.
	
	The control of $\mathcal{I}_1$ is slightly more involved. It requires the use of a classical commutator estimate, which can be found in \cite[Lemma 2.97]{bcd11} and gives that
	\begin{equation*}
		\norm{
		\left[ \Delta_{j+1}^\varepsilon, u \cdot \nabla  \right]
		\Delta_{j}^\varepsilon \omega
		}_{L^{\frac{2q}{q+1}}}
		\lesssim (2^j\Theta_\varepsilon)^{-1}
		\norm{\nabla u}_{L^q}\norm{\nabla\Delta_j^\varepsilon\omega}_{L^{\frac{2q}{q-1}}}
		\lesssim
		\norm{\nabla u}_{L^q}\norm{\Delta_j^\varepsilon\omega}_{L^{\frac{2q}{q-1}}}.
	\end{equation*}
	Combining this bound with another use of H\"older's inequality provides a suitable estimate on the first term in $\mathcal{I}_1$. The control of the second term in $\mathcal{I}_1$ is similar, which completes the proof of the Lemma.
\end{proof}

\subsection{Sketch of proof of Theorem \ref{thm3}}

As previously emphasized, the proof of Theorem \ref{thm3} can be derived from a straightforward adaptation of the method of proof of Theorem \ref{thm2}. For the sake of completeness, we summarize below its key elements:

\begin{enumerate}
	\item First, the convergence  $u^\varepsilon \to u$ in $L^\infty_tL^2$ is classical and can be found in \cite{CH96CPDE}.
	\item Then, the convergence of the low frequencies of vorticities  in $L^\infty_tL^2$ relies on the convergence of velocities in $L^\infty_t L^2$ established in the preceding step and is obtained by following, \emph{mutatis mutandis,} the procedure from Section \ref{section:low:freq}.
	\item Finally, it remains to establish the evanescence of the high frequencies of the viscous vorticities. This can be achieved by following the roadmap laid out in Section \ref{section:high:freq}. Indeed, the only difference between the settings of Theorems \ref{thm2} and \ref{thm3} becomes apparent in \eqref{energy-HF}, which, employing the same notation, needs now to be replaced by the estimate
	\begin{equation*}
	\begin{aligned}
	 \frac{1}{2 }\norm{  \sqrt{\mathrm{Id} - S_0^\varepsilon} \omega^\varepsilon (t )   }_{L^2}^2 &\leq \frac{1}{2 }\norm{  \sqrt{\mathrm{Id} - S_0^\varepsilon} \omega^\varepsilon_0 }_{L^2}^2
	 \\
	 &\quad+ \int_0^t  \int_{\mathbb{R}^2}  \curl g^\varepsilon  (\mathrm{Id} - S_0^\varepsilon) \omega^\varepsilon(\tau,x)  dxd\tau
	 -\mathcal{J}(t).
	\end{aligned}
	\end{equation*}
	The fact that \eqref{energy-HF} can be substituted by the preceding inequality is a direct consequence of the positivity of the dissipation produced by the diffusive term $-\varepsilon \Delta \omega^\varepsilon$ when performing energy estimates. The rest of the proof proceeds then with no further changes and allows us to complete the justification of Theorem \ref{thm3}. \qed
\end{enumerate}

\section{Rate of convergence of the inviscid limit in Yudovich's class}
\label{section:rate}

\subsection{Second main result}

Here, we show that our new method of proof of Theorem \ref{thm2} leads to a refined understanding of the inviscid limit in the two-dimensional Navier--Stokes system
\begin{equation}\label{NS*-2}
	\begin{cases}
		\begin{aligned}
 \partial_t u^{ \varepsilon} +u^{ \varepsilon}  \cdot\nabla u^{ \varepsilon}   - \varepsilon \Delta u^{ \varepsilon}   + \nabla p^{ \varepsilon}  &=0, 
			\\
			 \div u^{ \varepsilon}  &= 0,\\
			{ u^{ \varepsilon} } |_{t=0}  &=u_{0} .
		\end{aligned}
	\end{cases}
\end{equation}
Specifically, we prove the following result.

(In the statement below, we use the classical notation $B^s_{p,q}$ to denote the usual Besov space of regularity $s\in\mathbb{R}$, integrability $p\in[1,\infty]$ and frequency summability $q\in[1,\infty]$. We will also make use of the homogeneous version of Besov spaces defined in the usual way.)

\begin{thm}\label{thm-rate-general}
	Let $T>0$, $s\in (0,1)$ and $u_0\in H^1 \cap \dot{B}^{1+s}_{2,\infty} \cap  \dot{B}^{1+s}_{p,\infty}$, for some $p\in (2,\infty)$. Assume also that  $\omega_0 = \curl u_0\in L^\infty$. Let  $u^\varepsilon$  be the unique   solution  of \eqref{NS*-2} and $u$ be the unique Yudovich solution of the Euler equation \eqref{Euler} with the same initial data $u_0$ and no source term, i.e., $g=0$.
	
	Further assume that
	\begin{equation}\label{alpha:convergence}
		 \norm {u^\varepsilon-u}_{L^\infty([0,T];L^2)}=O\left( \varepsilon^{\alpha}\right),
	\end{equation}
	for some $\alpha>0$, as $ \varepsilon\to 0$, and that
	\begin{equation}\label{Besov-ES:general}
		 t\mapsto  \norm{(\omega^\varepsilon,\omega)(t)}_{B^{s(t) } _{2,\infty}\cap
		 B^{s(t) } _{p,\infty}(\mathbb{R}^2)   } \in L^\infty([0,T] ),
	\end{equation}
	uniformly in $\varepsilon$, for some nonnegative function $s(t)\in C^1([0,T])$, with $s(0)=s$ and $s'(t)<0$, for all $t\in [0,T]$.
	
	Then, it holds that
	$$
	\norm {u^\varepsilon-u}_{L^\infty([0,T];\dot{H}^1)}
	=O\left(
	\left(\frac{\varepsilon^{ 2\alpha s(T) }}
	{|\log \varepsilon| }\right)^{ \frac{1}{2(1 + s(T)) }} \right),
	$$
	as $\varepsilon\to 0.$
\end{thm}

	Observe that \eqref{Besov-ES:general} yields that $\omega^\varepsilon$ and $\omega$ are bounded in $L^\infty([0,T];\dot B^{s(T)}_{2,\infty})$, uniformly in $\varepsilon$. Therefore, employing the classical convexity inequality
	\begin{equation*}
		\begin{aligned}
			\norm {u^\varepsilon-u}_{L^\infty([0,T];\dot H^1)}
			&\lesssim
			\norm {u^\varepsilon-u}_{L^\infty([0,T];\dot B^1_{2,1})}
			\\
			&\lesssim
			\norm {u^\varepsilon-u}_{L^\infty([0,T];\dot B^{0}_{2,\infty})}^{\frac{s(T)}{1+s(T)}}
			\norm {u^\varepsilon-u}_{L^\infty([0,T];\dot B^{1+s(T)}_{2,\infty})}^{\frac{1}{1+s(T)}}
			\\
			&\lesssim
			\norm {u^\varepsilon-u}_{L^\infty([0,T];L^2)}^{\frac{s(T)}{1+s(T)}}
			\norm {u^\varepsilon-u}_{L^\infty([0,T];\dot B^{1+s(T)}_{2,\infty})}^{\frac{1}{1+s(T)}}
		\end{aligned}
	\end{equation*}
	along with the bounds \eqref{alpha:convergence} and \eqref{Besov-ES:general}
	yields the interpolated rate of convergence
	\begin{equation}\label{nonoptimal:rate}
		\norm {u^\varepsilon-u}_{L^\infty([0,T];\dot{H}^1)}=O\left( \varepsilon^{\frac{\alpha s(T)}{1+s(T)}}    \right).
	\end{equation}
	This rate of convergence has already been discussed by other authors in \cite[Corollary 2]{CDE}. The conclusion of Theorem \ref{thm-rate-general} establishes a subtle logarithmic refinement of the above rate of convergence as a consequence of the fact that the regularity of $\omega^\varepsilon$ and $\omega$ is strictly larger than $s(T)$ for times $t<T$. This is achieved by building on our new method of proof of Theorem \ref{thm2}.
	
	Thus, the assumption that $s'(t)$ be negative is essential and describes a ``memory effect'' of the higher regularity at previous times. Moreover, it seems plausible that this phenomenon is not exclusive to hydrodynamical models. It therefore deserves further investigation.

\begin{rem}
	A straightforward adaptation of the proof of convergence of velocities in $L^\infty _tL^2$ given in Section \ref{section:recall} shows that \eqref{alpha:convergence} holds with the value $\alpha=\frac 12\exp(-C_*T)$, where
		$C_* = C \norm {\omega_0}_{L^2\cap L^\infty}$,
	for some universal constant $C>0$. A precise justification for this result can be found in \cite[Theorem 7.37]{bcd11}.
	Furthermore, under the assumptions of Theorem \ref{thm-rate-general}, classical regularity results on transport flows (see \cite[Theorem 3.28]{bcd11}) show that the unique solution $u^\varepsilon$ to \eqref{NS*-2} enjoys the decaying regularity estimate \eqref{Besov-ES:general} with $s(t)=s\exp(-C_*t)$. Hence, an application of Theorem \ref{thm-rate-general} establishes the rate of convergence
	$$
	\norm {u^\varepsilon-u}_{L^\infty([0,T];\dot{H}^1)}
	=O\left(
	\left(\frac{\varepsilon^{s\exp(-2C_*T)}}
	{|\log \varepsilon| }\right)^{ \frac{1}{2(1+s\exp(-C_*T))}} \right),
	$$
	as $\varepsilon\to 0$. Note that there are other results (see \cite{brue21}, for instance) showing that \eqref{Besov-ES:general} holds with even slower regularity decays, thereby further improving the ensuing rates of convergence stemming from the application of Theorem \ref{thm-rate-general}.
\end{rem}

\begin{rem}
	Notably, as shown in \cite{CH96CPDE}, it is possible to construct solutions for a large class of velocities which do not have finite initial energy, i.e., such that $u_0\notin L^2$.
We believe that Theorem \ref{thm-rate-general} remains valid even when $u_0$ is unbounded in $L^2$. Indeed, the proof of the theorem only uses the fact that the difference $u^\varepsilon-u$ vanishes in $L^\infty_tL^2$, which is assumed to hold in \eqref{alpha:convergence}. This slight relaxation of the statement of Theorem \ref{thm-rate-general} would then allow for the consideration of vortex patches with non-smooth boundary, i.e., solutions which correspond to the initial data
$$\omega_0 = \mathds{1}_{\Omega},$$
where  $\Omega$ is a bounded domain in $\mathbb{R}^2$ whose boundary has a Minkowski dimension $D$ which is strictly less than $2$. Indeed, in that case,  one can show (see \cite[Lemma 3.2]{CW96}) that
$$\omega_0 \in B^{\frac{2-D}{p}} _{p,\infty}(\mathbb{R}^2),$$
for any $p\in [1,\infty)$.
Accordingly, it follows that
$$\omega_0 \in B^{s} _{2,\infty} \cap  B^{s} _{p,\infty}(\mathbb{R}^2) ,$$ 
for some $s\in (0,1)$.
\end{rem}

\subsection{Proof of Theorem \ref{thm-rate-general}}

This proof consists in refining the ideas leading to Theorem \ref{thm2} by exploiting the available additional regularity and optimizing the choice of parameters. We proceed in three steps:
\begin{enumerate}
	\item
	In Section \ref{rate:section:1}, we revisit the arguments laid out in Section \ref{section:low:freq}.
	\item
	In Section \ref{rate:section:2}, we exploit the additional regularity of solutions to improve the ideas behind the Compactness Extrapolation Lemma (Lemma \ref{lemma:extrapolation}).
	\item
	In Section \ref{rate:section:3}, we optimize the choice of $\Theta_\varepsilon$, which is an essential parameter describing the convergence of solutions.
\end{enumerate}

\subsubsection{Bounds on low frequencies}\label{rate:section:1}

We see that
\begin{equation*}
	\norm { \mathds{1}_{|D|\leq \Theta_\varepsilon}(\omega^\varepsilon - \omega )  }_{L^\infty([0,T]; L^2)}
	\lesssim
	\Theta_\varepsilon  \varepsilon^{\alpha},
\end{equation*}
for any parameter $\Theta_\varepsilon$, which will be chosen such that
\begin{equation}\label{theta:conditions}
	\lim_{\varepsilon \rightarrow 0}  \Theta_\varepsilon  = \infty
	\quad \text{ and }\quad
	\lim_{\varepsilon \rightarrow 0} \Theta_\varepsilon  \varepsilon^{\alpha}=0
\end{equation}
to ensure the convergence of low frequencies.

\subsubsection{Bounds on high frequencies}\label{rate:section:2}

We focus   now     on the evanescence of high frequencies. To that end, we write that
$$  \norm { \mathds{1}_{|D|\geq \Theta_\varepsilon}(\omega^\varepsilon - \omega )  }_{L^\infty([0,T]; L^2)} \leq  \norm { \mathds{1}_{|D|\geq \Theta_\varepsilon} \omega^\varepsilon    }_{L^\infty([0,T]; L^2)} + \norm { \mathds{1}_{|D|\geq \Theta_\varepsilon} \omega   }_{L^\infty([0,T]; L^2)} $$
and we estimate each term in the right-hand side, above, by refining the techniques laid out in Section \ref{section:high:freq} and exploiting the available additional regularity of velocities to obtain a sharp rate of convergence. We only  outline   the control of the high frequencies of $\omega^\varepsilon$, for the control of $\omega$ is identical.

From \eqref{energy-HF}, we find that
\begin{equation} \label{Energy-rate}
	\begin{aligned}
		\frac{1}{2 }\norm{  \sqrt{\mathrm{Id} - S_0^\varepsilon} \omega^\varepsilon (t )   }_{L^2}^2
		&=\frac{1}{2 }\norm{  \sqrt{\mathrm{Id} - S_0^\varepsilon} \omega _0 }_{L^2}^2
		-\varepsilon\norm{  \sqrt{\mathrm{Id} - S_0^\varepsilon} \omega^\varepsilon  }_{L^2_t \dot H^1}^2
		-\mathcal{J}(t)
		\\
		&\leq
		\frac{1}{2 }\norm{  \sqrt{\mathrm{Id} - S_0^\varepsilon} \omega _0 }_{L^2}^2
		+ \sum_{i=1}^3| \mathcal{ J}_i (t) |,
	\end{aligned}
\end{equation}
where  $\{ \mathcal{ J}_i\}_{i\in \{1,2,3\}}$ are given in \eqref{J:decomposition}.

Then, by applying Lemma \ref{T3-ES}, we find that 
 $$ \begin{aligned}
  |\mathcal{J}_1(t)|& \lesssim \int_0^t  \norm {\nabla u^\varepsilon (\tau)  }_{L^ q}
  \left(
  \norm {  \Delta_{-1}^\varepsilon \omega^\varepsilon  (\tau)    }_{L^\frac{2q}{q-1} }^2
  +
  \norm { \Delta_{0}^\varepsilon    \omega^\varepsilon  (\tau)    }_{L^\frac{2q}{q-1} }^2
  \right)
  d\tau  \\
  & \quad +  \int_0^t  \norm {\mathds{1}_{|D|>\frac{\Theta_\varepsilon }{6}}  \nabla u ^\varepsilon  (\tau) }_{L^2} \norm { \Delta_{-1}^\varepsilon \omega^\varepsilon  (\tau) }_{L^\infty} \norm{\Delta_{0}^\varepsilon\omega ^\varepsilon (\tau) }_{L^2}  d\tau,
 \end{aligned} $$
 where $q  \in (2,\infty)$ is fixed. It follows that 
$$ 
	|\mathcal{J}_1(t)|  \lesssim
	\norm {\omega ^\varepsilon }_{  L^\infty([0,t];  L^2\cap L^\infty) } \int_0^t    \norm {(\mathrm{Id}-S_{-3}^\varepsilon) \omega ^\varepsilon(\tau)}_{L^{\frac{2q}{q-1}}\cap L^{2} } ^2   d\tau,
$$
where we denoted
\begin{equation*}
	S_k^\varepsilon=\sum_{j\leq k-1}\Delta_j^\varepsilon,
\end{equation*}
for any $k\in\mathbb{Z}$. This takes care of $\mathcal{J}_1$.

As for $\mathcal{J}_2$ and $\mathcal{J}_3$, it is readily seen from \eqref{J2:ES1} and \eqref{J2:ES2} that
$$ 
	|\mathcal{J}_2(t)|+|\mathcal{J}_3(t)|  \lesssim
	\norm {\omega ^\varepsilon }_{  L^\infty([0,t]; L^\infty) } \int_0^t    \norm {(\mathrm{Id}-S_{-4}^\varepsilon) \omega ^\varepsilon(\tau)}_{ L^{2} }^2   d\tau.
$$
All in all, by plugging the foregoing bounds into \eqref{Energy-rate} and utilizing  \eqref{Transport:ES} with $p\in\{2,\infty\}$, we obtain that
  \begin{equation*} 
\norm{  \sqrt{\mathrm{Id} - S_0^\varepsilon} \omega^\varepsilon (t )   }_{L^2} ^2    \lesssim_{q} \norm{  \sqrt{\mathrm{Id} - S_0^\varepsilon} \omega _0 }_{L^2} ^2   +     \int_0^t\norm {(\mathrm{Id}- S_{-4}^\varepsilon ) \omega^\varepsilon(\tau)  }_{L^2 \cap L^{\frac{2q}{q-1}}} ^2d\tau.
\end{equation*} 
At last, observing that the same estimate holds for $\omega$ in place of $\omega^\varepsilon$, we deduce that
\begin{equation*} 
\begin{aligned}
\norm{  \sqrt{\mathrm{Id} - S_0^\varepsilon} \left(\omega^\varepsilon-\omega\right) (t )   }_{L^2} ^2
&
\lesssim_{q} \norm{  \sqrt{\mathrm{Id} - S_0^\varepsilon} \omega _0 }_{L^2} ^2   +     \int_0^t\norm {(\mathrm{Id}- S_{-4}^\varepsilon ) (\omega^\varepsilon,\omega)(\tau)  }_{L^2 \cap L^{\frac{2q}{q-1}}} ^2d\tau
\\
&\lesssim_q
\Theta_\varepsilon^{-2s}
\norm{   \omega _0 }_{\dot B_{2,\infty}^s} ^2
\\
&\quad+
\left(\int_0^t
\Theta_\varepsilon^{-2s(\tau)}d\tau\right)
\sup_{\tau\in[0,t]}\norm { (\omega^\varepsilon,\omega)(\tau)  }_{\dot{B}^{s(\tau)}_{2,\infty}  \cap   \dot{B}^{s(\tau)}_{\frac{2q}{q-1},\infty}} ^2,
\end{aligned}
\end{equation*}
for any $t\in[0,T]$. In particular, employing \eqref{Besov-ES:general}, we conclude that
\begin{equation*} 
\norm{  \sqrt{\mathrm{Id} - S_0^\varepsilon} \left(\omega^\varepsilon-\omega\right) (t )   }_{L^2} ^2
\lesssim
\Theta_\varepsilon^{-2s}
+
\int_0^t
\Theta_\varepsilon^{-2s(\tau)}d\tau,
\end{equation*}
where we chose $q$ so that $p=\frac{2q}{q-1}$ (without any loss of generality, we may assume here that $p\in (2,4)$).

Finally, evaluating that
$$
	\begin{aligned}
		\int_0^t  \Theta_\varepsilon^{-2 s(\tau)}d\tau
		&=
		\int_0^t\frac {-1}{2s'(\tau)\log\Theta_\varepsilon}
		\frac{d}{d\tau}\left(\Theta_\varepsilon^{- 2s(\tau)}\right)  d\tau
		\\
		&\leq
		\frac {-1}{2s'(t)\log\Theta_\varepsilon}
		\int_0^t
		\frac{d}{d\tau}\left(\Theta_\varepsilon^{- 2s(\tau)}\right)  d\tau
		\\
		&=\frac {-1}{2s'(t)\log\Theta_\varepsilon}
		\left(\Theta_\varepsilon^{- 2s(t)}-\Theta_\varepsilon^{- 2s}\right)
		\\
		&\lesssim\frac{\Theta_\varepsilon^{- 2s(t)}-\Theta_\varepsilon^{- 2s}}{\log\Theta_\varepsilon},
	\end{aligned}
$$
we arrive at
\begin{equation*} 
\norm{  \sqrt{\mathrm{Id} - S_0^\varepsilon} \left(\omega^\varepsilon-\omega\right) (t )   }_{L^2} ^2
\lesssim
\Theta_\varepsilon^{-2s}
+\frac{\Theta_\varepsilon^{- 2s(t)}}{\log\Theta_\varepsilon} .
\end{equation*}

\subsubsection{Optimization of $\Theta_\varepsilon$}\label{rate:section:3}

By combining the low-frequency and  high-frequency estimates established in the previous two steps, we obtain that
\begin{equation}\label{rate:finale:ES}
\norm{\omega^\varepsilon-\omega }_{L^\infty([0,T]; L^2)}
\lesssim
\Theta_\varepsilon^{-s}+
\Theta_\varepsilon   \varepsilon^{\alpha}
+\frac{\Theta_\varepsilon^{- s(T)}}{\log^\frac 12\Theta_\varepsilon}
\lesssim
\Theta_\varepsilon   \varepsilon^{\alpha}
+\frac{\Theta_\varepsilon^{- s(T)}}{\log^\frac 12\Theta_\varepsilon},
\end{equation}
where we employed that $s(T)<s$, for any fixed time $T>0$, to claim that
\begin{equation*}
	\Theta_\varepsilon^{-s}
	\lesssim
	\frac{\Theta_\varepsilon^{- s(T)}}{\log^\frac 12\Theta_\varepsilon}.
\end{equation*}
We are now going to choose a specific value for $\Theta_\varepsilon$ satisfying \eqref{theta:conditions} which will optimize the above estimate. 
To that end, notice first that the value 
 $$\Theta_\varepsilon = \left(\frac{s(T)}{\varepsilon^{\alpha}}\right)^{\frac {1}{1+ s(T)}} $$
 is the global minimum   of the function 
 $$\Theta_\varepsilon\mapsto
 \Theta_\varepsilon \varepsilon^{\alpha} +  \Theta_\varepsilon^{-s(T)}.$$
 Accordingly, if we   ignore the logarithmic correction in \eqref{rate:finale:ES}, this  choice of $\Theta_\varepsilon$ entails the rate of convergence displayed in \eqref{nonoptimal:rate}, which is not optimal.

In order to improve this rate of convergence, we need now to exploit the logarithmic decay of the last term in \eqref{rate:finale:ES}. Thus, we define
$$ \Theta_\varepsilon \bydef   \varepsilon^{-\frac{\alpha}{1+ s(T)}} |\log \varepsilon|^{-\beta},$$
where $\beta\in\mathbb{R}$ will be determined shortly. In particular, observe that
\begin{equation*}
	\log \Theta_\varepsilon\gtrsim |\log\varepsilon|,
\end{equation*}
as $\varepsilon\to 0$.
Therefore, incorporating this value into \eqref{rate:finale:ES} yields that
\begin{equation*}
	\begin{aligned}
		\norm{\omega^\varepsilon-\omega }_{L^\infty([0,T]; L^2)}
		&\lesssim
		\varepsilon^{\frac{\alpha s(T)}{1+ s(T)}}
		\left(|\log \varepsilon|^{-\beta}
		+|\log \varepsilon|^{\beta s(T)}\log^{-\frac 12}\Theta_\varepsilon
		\right)
		\\
		&\lesssim
		\varepsilon^{\frac{\alpha s(T)}{1+ s(T)}}
		\left(|\log \varepsilon|^{-\beta}
		+|\log \varepsilon|^{\beta s(T)-\frac 12}
		\right).
	\end{aligned}
\end{equation*}
At last, optimizing the value of $\beta$ by setting
\begin{equation*}
	\beta\bydef \frac{1}{2(1+s(T))}
\end{equation*}
leads to the rate of convergence claimed in the statement of Theorem \ref{thm-rate-general}, which completes its proof.
\qed

\section{Two-dimensional incompressible plasmas}
\label{section:plasma}

\subsection{Third main result}

We are now interested in two-dimensional incompressible plasmas described by the Euler--Maxwell equations
\begin{equation}\label{EM}
	\begin{cases}
		\begin{aligned}
			\text{\tiny(Euler's equation)}&&&\partial_t u +u \cdot\nabla u   =  - \nabla p + j \times B, &\div u =0,&
			\\
			\text{\tiny(Amp\`ere's equation)}&&&\frac{1}{c} \partial_t E - \nabla \times B =- j , &\div E = 0,&
			\\
			\text{\tiny(Faraday's equation)}&&&\frac{1}{c} \partial_t B + \nabla \times E  = 0 , &\div B = 0,&
			\\
			\text{\tiny(Ohm's law)}&&&j= \sigma \big( cE + P(u \times B)\big), &\div j = 0,&
		\end{aligned}
	\end{cases}
\end{equation}
where $(t,x)\in [0,\infty)\times \mathbb{R}^2$ and $P= \mathrm{Id}-\Delta^{-1}\nabla\mathrm{div}$ denotes Leray's projector onto divergence-free vector fields, supplemented with initial data $(u,E,B)_{|t=0}=(u_0,E_0,B_0)$ such that $\div u_0=\div E_0=\div B_0=0$.

The positive parameters $\sigma$ and $c$ refer to the electric conductivity and the speed of light, respectively, whereas the physical variables $u(t,x)$, $E(t,x)$ and $B(t,x)$ are vector fields with values in $\mathbb{R}^3$ which are interpreted as the fluid velocity, the electric field and the magnetic field, respectively. Finally, the field $j(t,x)$ denotes the electric current density and is determined by Ohm's law.

A straightforward calculation shows that solutions to the above system satisfy, at least formally, the energy inequality
\begin{equation}\label{energy-inequa}
	\norm {u(t)}_{L^2}^2 + \norm {E(t)}_{L^2}^2 + \norm {B(t)}_{L^2}^2
	+\frac{2}{\sigma}\int_0^t \norm {j(\tau)}_{L^2}^2 d\tau \leq \mathcal{E}_0^2,
\end{equation}
for all $t\geq 0$, where we denote
\begin{equation*}
	\mathcal{E}_0 \bydef \left( \norm {u_0}_{L^2}^2 + \norm {E_0}_{L^2}^2 + \norm {B_0}_{L^2}^2 \right)^{\frac{1}{2}}.
\end{equation*}
Current methods fail to produce the existence of solutions to \eqref{EM} for initial data merely satisfying that $\mathcal{E}_0$ is finite. In general, when studying solutions of \eqref{EM}, it is therefore natural to suppose that $(u_0,E_0,B_0)$ enjoys some suitable enhanced regularity.

Henceforth, we further assume that the fields $u$, $E$ and $B$ satisfy the two-dimensional normal structure
\begin{equation}\label{structure:2dim}
	u(t,x)=
	\begin{pmatrix}
		u_1(t,x)\\u_2(t,x)\\0
	\end{pmatrix},
	\qquad
	E(t,x)=
	\begin{pmatrix}
		E_1(t,x)\\E_2(t,x)\\0
	\end{pmatrix}
	\qquad\text{and}\qquad
	B(t,x)=
	\begin{pmatrix}
		0\\0\\b(t,x)
	\end{pmatrix}.
\end{equation}
This structure played a significant role in our recent work \cite{arsenio:houamed:2022}, where we established the global existence and uniqueness of solutions to \eqref{EM}, with the structure \eqref{structure:2dim}, provided the speed of light $c$ is sufficiently large when compared to the size of the initial data in suitable norms. The condition therein relating $c$ to the initial data can be interpreted as a strengthening of the physical principle that the velocity of the fluid cannot exceed the speed of light. A precise reformulation of the well-posedness result from \cite{arsenio:houamed:2022} is contained in the following theorem.

\begin{thm}[\cite{arsenio:houamed:2022}]\label{thm-EM0}
	Let $s$ be any real number in $(\frac 74,2)$ and consider a bounded family of initial data
	\begin{equation*}
		\big\{(u_0^c,E_0^c,B_0^c)\big\}_{c>0}\subset
		\left(H^1\times H^s \times H^s \right)(\mathbb{R}^2),
	\end{equation*}
	with $\div u_0^c=\div E_0^c=\div B_0^c$ and the two-dimensional normal structure \eqref{structure:2dim}, such that the initial vorticities $\omega_0^c\bydef \nabla\times u_0^c$ form a bounded family of $L^\infty(\mathbb{R}^2)$.
	
	There is a constant $c_0>0$, such that, for any speed of light $c\in(c_0,\infty)$, there is a global weak solution $(u^c,E^c,B^c)$ to the two-dimensional Euler--Maxwell system \eqref{EM}, with the normal structure \eqref{structure:2dim} and initial data $(u_0^c,E_0^c,B_0^c)$, satisfying the energy inequality \eqref{energy-inequa} and enjoying the additional regularity
	\begin{equation}\label{propagation:damping:2}
		\begin{gathered}
			u^c\in L^\infty(\mathbb{R}^+; L^\infty\cap H^1),
			\quad
			\omega^c\bydef \nabla\times u^c \in L^\infty(\mathbb{R}^+;L^\infty),
			\quad
			(E^c,B^c)\in L^\infty(\mathbb{R}^+; H^s),
			\\
			(cE^c,B^c)\in L^2(\mathbb{R}^+;\dot H^1\cap\dot H^s),
			\quad
			(E^c,B^c)\in L^2(\mathbb{R}^+;\dot W^{1,\infty}),
			\quad
			j^c\in L^2(\mathbb{R}^+;L^2\cap L^\infty),
		\end{gathered}
	\end{equation}
	where $j^c\bydef \sigma \big( cE^c + P(u^c \times B^c)\big)$.
	It is to be emphasized that the bounds in \eqref{propagation:damping:2} are uniform in $c\in(c_0,\infty)$, for any bounded family of initial data.
	
	Furthermore, for each $c\in(c_0,\infty)$, the solution $(u^c,E^c,B^c)$ is unique in the space of all solutions $(\bar u, \bar E,\bar B)$ to the Euler--Maxwell system \eqref{EM} satisfying the bounds, locally in time,
	\begin{equation*}
		(\bar u, \bar E,\bar B)\in L^\infty_tL^2_x,
		\qquad \bar u \in L^{2}_tL^\infty_x,
		\qquad \bar j \in L^{2}_{t,x},
	\end{equation*}
	and having the same initial data.
\end{thm}

\begin{rem}
	We should note that the uniform bound $j^c\in L^2(\mathbb{R}^+;L^\infty)$ in \eqref{propagation:damping:2} cannot be found in the statements of the theorems from \cite{arsenio:houamed:2022}. However, it is established explicitely in Section 3.8 therein, as an intermediate step of the proof of the uniform bound $\omega^c\in L^\infty(\mathbb{R}^+;L^\infty)$.
\end{rem}

In view of the uniform bounds \eqref{propagation:damping:2}, it is possible to establish the convergence
\begin{equation*}
	(u^c,E^c,B^c)\stackrel{c\to\infty}{\longrightarrow}(u,0,B),
\end{equation*}
in the strong topology of $L^2_{\mathrm{loc}}(\mathbb{R}^+\times\mathbb{R}^2)$, of the solution from Theorem \ref{thm-EM0}, where $(u,B)$ is the unique solution to the magnetohydrodynamic system
\begin{equation}\label{mhd}
	\left\{
	\begin{aligned}
		&\partial_t u +u \cdot\nabla u = -\nabla p , &\div u =0,&
		\\ 
		&\partial_t B + u \cdot\nabla B - \frac{1}{\sigma}\Delta B = 0,&&
	\end{aligned}
	\right.
\end{equation}
corresponding to the initial data $(u_0,B_0)$ and satisfying the bounds
\begin{equation*}
	\begin{gathered}
		u\in L^\infty(\mathbb{R}^+; L^\infty\cap H^1),
		\quad
		\omega\bydef \nabla\times u \in L^\infty(\mathbb{R}^+;L^\infty),
		\\
		B\in L^\infty(\mathbb{R}^+; H^1)
		\cap
		L^2(\mathbb{R}^+;\dot H^1\cap\dot H^2\cap\dot W^{1,\infty}).
	\end{gathered}
\end{equation*}
This convergence result and the ensuing bounds follow from a standard compactness argument and is formulated in Corollary 1.2 from \cite{arsenio:houamed:2022}.

Our aim is now to improve on this asymptotic characterization of \eqref{EM} by showing that the convergence holds globally in space, uniformly in time, in adequate functional spaces. Our main global convergence result for inviscid two-dimensional plasmas is contained in the following theorem.

\begin{thm}\label{thm-EM}
	Let $\big\{(u_0^c,E_0^c,B_0^c)\big\}_{c>0}$ be a bounded family of initial data as required in Theorem \ref{thm-EM0}, for some fixed $s\in(\frac 74,2)$, and consider the unique solution $(u^c,E^c,B^c)$ to \eqref{EM} provided by the same theorem, for large $c$. If $ u_0^c $ converges to $u_0$ in ${H}^1(\mathbb{R}^2)$ and $(E_0^c,B_0^c) $ converges to $(0,B_0)$ in $  L^2(\mathbb{R}^2) $, then  $(u^c, B^c)$ converges strongly to $(u,B)$, the unique global solution of \eqref{mhd}, on any finite time interval $[0,T]$, in the sense that
	$$
	\lim_{c\rightarrow \infty}  \norm {u^c-u }_{L^\infty([0,T]; H^1(\mathbb{R}^2))}  = 0,
	$$
	and 
	$$
	\lim_{c\rightarrow \infty}
	\left(\norm {B^c-B }_{L^\infty([0,T]; L^2(\mathbb{R}^2))}
	+ \norm {B^c-B }_{L^2([0,T]; \dot{H}^1(\mathbb{R}^2))} \right)= 0.
	$$
	Furthermore, the electric current density $j^c=\sigma \big( cE^c + P(u^c \times B^c)\big)$ converges strongly to $\nabla\times B$ in the sense that
	\begin{equation*}
		\lim_{c\to\infty}\norm {j^c-\nabla\times B }_{L^2([0,T]; L^2(\mathbb{R}^2))}=0.
	\end{equation*}
\end{thm}

\begin{rem}
	It is possible to quantify the convergence of $u^c$ and $B^c$ in the preceding theorem with a rate $O(c^{-\alpha})$, for some $\alpha>0$. This will be clear in the proof of the theorem, below.
	Moreover, the electric field $E^c$ vanishes in $L^2(\mathbb{R}^+;\dot H^1)$ with a rate $O(c^{-1})$ due to the uniform bounds given in \eqref{propagation:damping:2}.
\end{rem}

\begin{rem}
	In the statements of the above results, the restriction on the range of the regularity parameter $s\in (\frac 74, 2)$ is only necessary in the construction of global solutions given in Theorem \ref{thm-EM0}. However, Theorem \ref{thm-EM} would hold for the wider range $s\in [1,2]$ provided solutions to \eqref{EM} satisfying the uniform bounds \eqref{propagation:damping:2} could be constructed for that range.
\end{rem}

The proof of Theorem \ref{thm-EM} is given in Section \ref{sec.proof THM}, below. It relies on a self-contained asymptotic analysis of Amp\`ere's equation, which is detailed in the next section.

\subsection{Asymptotic analysis of Amp\`ere's equation}\label{sec:ampere equation}

A formal asymptotic analysis of Amp\`ere's equation readily shows that $j^c$ converges to $\nabla\times B$, in the regime $c\to\infty$. The proof of our main result Theorem \ref{thm-EM} hinges upon a refined asymptotic analysis of the same equation. More precisely, the following proposition provides a key estimate on the distance between $\nabla\times B^c$ and $j^c$, as $c\to\infty$.

\begin{prop}\label{prop:CV}
	Let $\big\{(u_0^c,E_0^c,B_0^c)\big\}_{c>0}$ be a bounded family of initial data as required in Theorem \ref{thm-EM0}, for some fixed $s\in(\frac 74,2)$, and consider the unique solution $(u^c,E^c,B^c)$ provided by the same theorem, for $c>c_0$. Then, one has the estimate
	\begin{equation*}
		\norm {\nabla\times B^c-j^c}_{L^2(\mathbb{R}^+;\dot{H}^{\eta-1})}
		\lesssim
		\frac 1c \norm{\nabla\times B^c_0-j^c_0}_{\dot{H}^{\eta-1}}
		+\frac 1c\norm{E^c_0}_{\dot{H}^\eta}
		+\frac {C_*}{c^2},
	\end{equation*}
	for every $\eta\in [1,s]$ and $c>c_0$, where the initial electric current density is given by
	\begin{equation*}
		j_0^c\bydef \sigma\left( cE_0^c + P(u_0^c\times B_0^c) \right)
	\end{equation*}
	and $C_*>0$ depends on the family of initial data.
\end{prop}

\begin{rem}
	Notice that the family of initial data considered in the preceding proposition is bounded uniformly in the spaces
	\begin{equation*}
		u_0^c\in L^2,
		\qquad
		E_0^c\in H^1,
		\qquad
		B_0^c\in \dot H^1\cap L^\infty.
	\end{equation*}
	Therefore, setting $\eta=1$, the above result provides us with the estimate
	\begin{equation}\label{E-NL:es}
		\norm {\nabla\times B^c-j^c}_{L^2(\mathbb{R}^+;L^2)}
		\lesssim
		\norm{E^c_0}_{L^2}
		+\frac 1c,
	\end{equation}
	which shows that $\nabla\times B^c-j^c$ vanishes in the topology of $L^2$ if the initial electric field converges strongly to zero.
\end{rem}

\begin{proof}
	First applying a time derivative to Amp\`ere's equation and combining it with Faraday's equation and Ohm's law yields the damped wave equation 
	\begin{equation*}
		\frac{1}{c} \partial_t^2E^c - c \Delta E^c + \sigma c \partial_t E^c = -\sigma \partial_t P (u^c\times B^c).
	\end{equation*}
	Then, performing a classical energy estimate by multiplying the above damped wave equation by $(-\Delta)^{\eta-1}\partial_t E^c$ and integrating in space (this can be done rigorously by employing a Littlewood--Paley decomposition, for instance), we find that
	\begin{equation*}
		\begin{aligned}
		\frac 12	\frac{d}{dt}\left( \frac{1}{c} \norm{\partial_t E^c (t)}_{\dot{H}^{\eta-1}}^2
			+c \norm{E^c (t)}_{\dot{H}^\eta}^2 \right)
			&+\sigma c\norm {\partial_t E^c(t)}_{\dot{H}^{\eta-1}}^2
			\\
			&\leq \sigma
			\norm{\partial_t P (u^c\times B^c)(t)}_{\dot{H}^{\eta-1}}
			\norm{\partial_tE^c(t)}_{\dot{H}^{\eta-1}}
			\\
			&\leq \frac{\sigma}{2c}
			\norm{\partial_t P (u^c\times B^c)(t)}_{\dot{H}^{\eta-1}}^2
			+\frac{\sigma c}2\norm{\partial_tE^c(t)}_{\dot{H}^{\eta-1}}^2,
		\end{aligned}
	\end{equation*}
	whereby
	\begin{equation*}
	 	\frac{d}{dt}\left( \frac{1}{c} \norm{\partial_t E^c (t)}_{\dot{H}^{\eta-1}}^2
		+c \norm{E^c (t)}_{\dot{H}^\eta}^2 \right)
		+ \sigma c \norm {\partial_t E^c(t)}_{\dot{H}^{\eta-1}}^2
		\leq \frac{\sigma}{ c}
		\norm{\partial_t P (u^c\times B^c)(t)}_{\dot{H}^{\eta-1}}^2.
	\end{equation*}
	Further integrating in time, we deduce that
	\begin{equation}\label{Energy***}
		\begin{aligned}
			\frac{1}{c} \norm{\partial_t E^c}_{L^\infty(\mathbb{R}^+;\dot{H}^{\eta-1})}
			+&\norm{E^c}_{L^\infty(\mathbb{R}^+;\dot{H}^\eta)}
			+\norm {\partial_t E^c}_{L^2(\mathbb{R}^+;\dot{H}^{\eta-1})}
			\\
			&\lesssim
			\frac{1}{c} \norm{\partial_t E^c_0}_{\dot{H}^{\eta-1}}
			+\norm{E^c_0}_{\dot{H}^\eta}
			+\frac 1{c}
			\norm{\partial_t P (u^c\times B^c)}_{L^2(\mathbb{R}^+;\dot{H}^{\eta-1})}.
		\end{aligned}
	\end{equation}
	
	We are now left with estimating $ \partial_t P (u^c\times B^c)$, above. To this end, we exploit Euler and Faraday's equations from \eqref{EM} to write
	\begin{equation*}
		\begin{aligned}
			\partial_t P( u^c \times B^c)
			&=P\big( \partial_t u^c\times B^c + u^c \times \partial_t B^c\big)
			\\
			&= P\Big(\big(P(j^c \times B^c) - P (u^c\cdot\nabla  u^c)\big)\times B^c - c u^c \times (\nabla \times E^c)\Big)
			= \mathcal{I}_1+\mathcal{I}_2+\mathcal{I}_3,
		\end{aligned}
	\end{equation*}
	where
	\begin{equation*}
		\begin{aligned}
			\mathcal{I}_1&=P\Big(\big(P(j^c \times B^c) \big)\times B^c\Big),
			\\
			\mathcal{I}_2&=-P\Big(\big(P (u^c\cdot\nabla  u^c)\big)\times B^c\Big),
			\\
			\mathcal{I}_3&=-cP\Big(u^c \times (\nabla \times E^c)\Big).
		\end{aligned}
	\end{equation*}
	Next, we utilize the paradifferential product laws for normal vector fields established in \cite{arsenio:houamed:2022} to control each term individually. This is a step where the normal structure \eqref{structure:2dim} plays a crucial role. For convenience, we have gathered the relevant product estimates from \cite{arsenio:houamed:2022} in the appendix.

	Thus, by Lemma \ref{lemma:paradiff} and the energy inequality \eqref{energy-inequa}, we obtain that
	\begin{equation*}
		\begin{aligned}
			\norm{\mathcal{I}_1}_{L^2_{t}\dot{H}^{\eta-1}}
			&\lesssim
			\norm{P(j^c \times B^c)}_{L^2_t\dot{H}^{\frac{\eta-1}{2}}}
			\norm{B^c}_{L^\infty_t\dot{H}^{\frac{\eta+1}{2}}}\\
			&\lesssim \norm{ j^c }_{L^2_{t}L^2} \norm{ B^c }_{L^\infty_t \dot{H}^{\frac{\eta+1}{2}}}^2
			\lesssim \mathcal{E}_0\norm{ B^c }_{L^\infty_t {H}^{\eta}}^2.
		\end{aligned}
	\end{equation*}
	As for the second term $\mathcal{I}_2$, by applying Lemma \ref{lemma:paradiff}, we find that
	\begin{equation*}
		\begin{aligned}
			\norm {\mathcal{I}_2}_{L^2_{t}\dot{H}^{\eta-1}}
			&\lesssim \norm { u^c \cdot \nabla u^c}_{L^\infty _tL^2}
			\norm{B^c}_{L^2_t\dot{H}^{\eta}}
			\\
			&\lesssim \norm{ u^c }_{L^\infty _tL^\infty}
			\norm { u^c }_{L^\infty _t\dot{H}^1}
			\norm{B^c}_{L^2_t\dot{H}^{\eta}}.
		\end{aligned}
	\end{equation*}
	Finally, for the last term $\mathcal{I}_3$, using Lemma \ref{lemma:paradiff}, again, yields that
	\begin{equation*}
		\begin{aligned}
			\norm {\mathcal{I}_3}_{L^2_{t}\dot{H}^{\eta-1}} 
			&\lesssim \norm { u^c }_{L^\infty _t(L^\infty \cap \dot{H}^1)}
			\norm{c \nabla \times E^c}_{L^2_t\dot{H}^{\eta-1}}
			\\
			&\lesssim \norm { u^c }_{L^\infty _t(L^\infty \cap \dot{H}^1)}
			\norm{c E^c}_{L^2_t\dot{H}^{\eta}}.
		\end{aligned}
	\end{equation*}
	All in all, in view of the uniform bounds \eqref{propagation:damping:2}, we deduce that the control
	\begin{equation*}
		\partial_t P( u^c \times B^c)\in L^2(\mathbb{R}^+;\dot{H}^{\eta-1}), \quad \forall \eta \in [1,s],
	\end{equation*}
	holds uniformly in $c$.
	
	Now, by incorporating the control of $\partial_t P( u^c \times B^c)$ into \eqref{Energy***}, we find that
	\begin{equation*}
		\frac{1}{c} \norm{\partial_t E^c}_{L^\infty(\mathbb{R}^+;\dot{H}^{\eta-1})}
		+\norm{E^c}_{L^\infty(\mathbb{R}^+;\dot{H}^\eta)}
		+\norm {\partial_t E^c}_{L^2(\mathbb{R}^+;\dot{H}^{\eta-1})}
		\lesssim
		\frac{1}{c} \norm{\partial_t E^c_0}_{\dot{H}^{\eta-1}}
		+\norm{E^c_0}_{\dot{H}^\eta}
		+\frac {C_*}{c}.
	\end{equation*}
	At last, employing Amp\`ere's equation to substitute $\frac 1c\partial_t E^c$ with $\nabla\times B^c-j^c$ completes the proof.
\end{proof}

\subsection{Proof of Theorem \ref{thm-EM}}\label{sec.proof THM}

We proceed in two steps. First, we show the strong global convergence of the velocity field $u^c$ through an application of Theorem \ref{thm2} combined with Proposition \ref{prop:CV}. Then, we show the stability of the magnetic field $B^c$ by performing a suitable energy estimate on a transport--diffusion equation with a vanishing source term.

\subsubsection{Stability of the velocity field}

Employing the normal structure \eqref{structure:2dim}, we write the momentum equation from \eqref{EM} as
 $$ \partial_t u^c +u^c \cdot\nabla u^c    +  \nabla \left( p^c + \frac{|B^c|^2}{2}-(\mathrm{Id}-P)g^c\right) =  Pg^c,$$
 where 
 $$ g^c \bydef  \left(j^c- \nabla \times B^c\right) \times B^c.$$
We want now to apply Theorem \ref{thm2} to the above system to deduce the convergence of $u^c$, as $ c\rightarrow \infty$. To that end, we need to establish, first, that $ Pg^c $ vanishes in $ L^1 ([0,T];H^1(\mathbb{R}^2))$ and, second, that $\curl g^c$ remains uniformly bounded in $L^1([0,T];L^\infty(\mathbb{R}^2))$.

In order to prove the vanishing of $Pg^c$, we employ the paradifferential calculus estimate from Lemma \ref{lemma:paradiff} to find that
 $$   \norm {Pg^c}_{L^1_tL^2} \lesssim \norm { \nabla \times B^c - j^c}_{L^2_{t}L^2} \norm {B^c}_{L^2_t\dot{H}^1} .$$
 Furthermore, noticing that
\begin{equation*}
	\curl g^c = \left( \nabla \times B^c - j^c\right)  \cdot \nabla  B^c
	= - j^c  \cdot \nabla  B^c,
\end{equation*}
we see, by H\"older's inequality, that
  $$   \norm { P g^c}_{L^1_t\dot H^1}\lesssim\norm { \curl g^c}_{L^1_tL^2} \lesssim \norm { \nabla \times B^c - j^c}_{L^2_{t}L^2} \norm {B^c}_{L^2_t \dot{W}^{1,\infty}} .$$
  Hence, since the assumptions of Theorem \ref{thm-EM} guarantee that $E^c_0$ vanishes in $L^2$ and that $B^c$ remains bounded in $L^2_t(\dot H^1\cap \dot W^{1,\infty})$, we deduce from \eqref{E-NL:es}  that $Pg^c\to 0$ in $ L^1 ([0,T];H^1(\mathbb{R}^2))$.

As for the uniform bound of $\curl g^c$ in $ L^1 ([0,T];L^\infty(\mathbb{R}^2))$, it follows immediately upon noticing that both $j^c$ and $\nabla B^c$ are uniformly bounded in $ L^2 ([0,T];L^\infty(\mathbb{R}^2))$, as emphasized in \eqref{propagation:damping:2}.

We can now apply Theorem \ref{thm2} to conclude that
\begin{equation}\label{velocity:convergence}
	\lim_{c\rightarrow\infty}\sup_{t\in [0,T]} \norm {u^c(t)-u(t)} _{H^1(\mathbb{R}^2)} =0,
\end{equation}
thereby completing the proof of stability of the velocity field.

\subsubsection{Stability of the magnetic field.}

By suitably combining Faraday's equation with Ohm's law from \eqref{EM}, we observe that
\begin{equation*}
	\partial_t B^c+u^c\cdot\nabla B^c-\frac 1\sigma \Delta B^c=\frac 1\sigma\nabla\times(\nabla\times B^c- j^c).
\end{equation*}
Then, subtracting this equation with \eqref{mhd}, we see that
\begin{equation*} 
\partial_t \widetilde{B}^c + u^c\cdot \nabla \widetilde{B}^c - \frac 1\sigma \Delta \widetilde{B}^c = - (u^c- u)\cdot \nabla B+ \frac 1\sigma \nabla \times   ( \nabla \times B^c - j^c), 
\end{equation*}
where $\widetilde{B}^c\bydef B^c-B$,
with the initial data $ \widetilde{B}^c_0\bydef B^c_0-B_0$.

By a classical $L^2$-energy estimate, we find that
\begin{equation*}
	\begin{aligned}
		\|\widetilde{B}^c\|_{L^\infty([0,T];L^2)} &+ \|\widetilde{B}^c\|_{L^2([0,T];\dot{H}^1)}
		\\
		&\lesssim
		\|\widetilde{B}^c_0\|_{  L^2}
		+
		\left\|\nabla\times\left((u^c- u)\times B
		+ \frac 1\sigma
		( \nabla \times B^c - j^c)\right)\right\|_{L^2([0,T];\dot{H}^{-1})}
		\\
		&\lesssim
		\|\widetilde{B}^c_0\|_{  L^2}
		+  \norm {u^c- u}_{L^\infty([0,T];L^2)} \norm B_{ {L^2([0,T];\dot{H}^{1})}}
		+ \norm {\nabla \times B^c - j^c }_{{L^2([0,T];L^2)}},
	\end{aligned}
\end{equation*}
where we applied Lemma \ref{lemma:paradiff} to deduce the obtain estimate
\begin{equation*}
	\begin{aligned}
		\left\|\nabla\times\left((u^c- u)\times B\right)
		\right\|_{L^2([0,T];\dot{H}^{-1})}
		&\lesssim
		\left\|P\left((u^c- u)\times B\right)
		\right\|_{L^2([0,T];L^2)}
		\\
		&\lesssim
		\norm {u^c- u}_{L^\infty([0,T];L^2)} \norm B_{ {L^2([0,T];\dot{H}^{1})}}.
	\end{aligned}
\end{equation*}
It then follows from \eqref{E-NL:es} and \eqref{velocity:convergence} that
\begin{equation*}
	\lim_{c\rightarrow \infty}\left( \|\widetilde{B}^c \|_{L^\infty([0,T];  L^2)} + 	\|\widetilde{B}^c\|_{L^2([0,T]; \dot{H}^1)}\right) = 0,
\end{equation*}
which completes the proof of stability of the magnetic field.

At last, writing
$$ \norm {j^c - \nabla \times B}_{L^2([0,T]; L^2) } \leq
\norm {B^c -   B}_{L^2([0,T]; \dot H^1) } + \norm {j^c - \nabla \times B^c}_{L^2([0,T]; L^2) },$$
it is readily seen that $j^c$ converges to $\nabla\times B$ in $L^2([0,T]; L^2)$, which completes the proof of the theorem. \qed

\appendix

\section{Paradifferential calculus and the normal structure}

	It was noticed in \cite{arsenio:houamed:2022} that the normal structure \eqref{structure:2dim} of solutions to the Euler--Maxwell system could be exploited to extend the classical product laws of paradifferential calculus to a wider range of parameter. We recall here a simplified statement of Lemma 3.4 from \cite{arsenio:houamed:2022} which provides the product laws which are relevant to the present work. A complete justification of this result can be found in \cite{arsenio:houamed:2022}.  
	
	\begin{lem}\label{lemma:paradiff}
		Let $s\in (-1,2)$ and let $F,G: \mathbb{R}^2 \rightarrow \mathbb{R}^3$ be divergence-free vector fields having the normal structure
		\begin{equation*} 
			F(x)=
			\begin{pmatrix}
				F_1( x)\\F_2( x)\\0
			\end{pmatrix}
			\qquad\text{and}\qquad
			G(x)=
			\begin{pmatrix}
				0\\0\\G_3(x)
			\end{pmatrix}.
		\end{equation*}
		Then, for any $\eta\in (-\infty,1)$ and $s\in (-\infty,2)$, with $\eta + s >0$, one has the product law
		\begin{equation*}
			\norm{P(F\times G)}_{\dot{H}^{\eta+ s -1}} \lesssim \norm F_{\dot{H}^\eta} \norm G_{\dot{H}^s}.
		\end{equation*}
		Moreover, in the endpoint case $\eta=1$, one has that
		\begin{equation*}
			\norm{P(F\times G)}_{\dot{H}^{s }} \lesssim \norm F_{L^\infty \cap \dot{H}^1} \norm G_{\dot{H}^s},
		\end{equation*}
		for any $s\in (-1,2)$.
\end{lem}


\bibliographystyle{plain} 
\bibliography{BIB.bib}

\begin{thebibliography}{10}

\bibitem{AD04}
H.~Abidi and R.~Danchin.
\newblock Optimal bounds for the inviscid limit of {N}avier-{S}tokes equations.
\newblock {\em Asymptot. Anal.}, 38(1):35--46, 2004.

\bibitem{arsenio:houamed:2022}
Diogo Ars\'enio and Haroune Houamed.
\newblock Damped {S}trichartz estimates and the incompressible
  {E}uler--{M}axwell system.
\newblock {\em arXiv}, 2022.
\newblock \url{https://arxiv.org/abs/2204.04277}.

\bibitem{bcd11}
Hajer Bahouri, Jean-Yves Chemin, and Rapha\"{e}l Danchin.
\newblock {\em Fourier analysis and nonlinear partial differential equations},
  volume 343 of {\em Grundlehren der Mathematischen Wissenschaften [Fundamental
  Principles of Mathematical Sciences]}.
\newblock Springer, Heidelberg, 2011.

\bibitem{BT13}
Claude~W. Bardos and Edriss~S. Titi.
\newblock Mathematics and turbulence: where do we stand?
\newblock {\em J. Turbul.}, 14(3):42--76, 2013.

\bibitem{BM81}
J.~Thomas Beale and Andrew Majda.
\newblock Rates of convergence for viscous splitting of the {N}avier-{S}tokes
  equations.
\newblock {\em Math. Comp.}, 37, 1981.

\bibitem{brue21}
Elia Bru\'{e} and Quoc-Hung Nguyen.
\newblock Sobolev estimates for solutions of the transport equation and {ODE}
  flows associated to non-{L}ipschitz drifts.
\newblock {\em Math. Ann.}, 380(1-2):855--883, 2021.

\bibitem{CH96CPDE}
Jean-Yves Chemin.
\newblock A remark on the inviscid limit for two-dimensional incompressible
  fluids.
\newblock {\em Comm. Partial Differential Equations}, 21(11-12):1771--1779,
  1996.

\bibitem{CCS}
Gennaro Ciampa, Gianluca Crippa, and Stefano Spirito.
\newblock Strong convergence of the vorticity for the 2{D} {E}uler equations in
  the inviscid limit.
\newblock {\em Arch. Ration. Mech. Anal.}, 240(1):295--326, 2021.

\bibitem{C86}
Peter Constantin.
\newblock Note on loss of regularity for solutions of the {$3$}-{D}
  incompressible {E}uler and related equations.
\newblock {\em Comm. Math. Phys.}, 104(2):311--326, 1986.

\bibitem{CDE}
Peter Constantin, Theodore~D. Drivas, and Tarek~M. Elgindi.
\newblock Inviscid limit of vorticity distributions in the {Y}udovich class.
\newblock {\em Comm. Pure Appl. Math.}, 75(1):60--82, 2022.

\bibitem{CW95}
Peter Constantin and Jiahong Wu.
\newblock Inviscid limit for vortex patches.
\newblock {\em Nonlinearity}, 8(5):735--742, 1995.

\bibitem{CW96}
Peter Constantin and Jiahong Wu.
\newblock The inviscid limit for non-smooth vorticity.
\newblock {\em Indiana Univ. Math. J.}, 45(1):67--81, 1996.

\bibitem{DL89}
R.~J. DiPerna and P.-L. Lions.
\newblock Ordinary differential equations, transport theory and {S}obolev
  spaces.
\newblock {\em Invent. Math.}, 98(3):511--547, 1989.

\bibitem{HHZ}
Taoufik Hmidi, Haroune Houamed, and Mohamed Zerguine.
\newblock Rigidity aspects of singular patches in stratified flows.
\newblock {\em Tunis. J. Math.}, 4(3):465--557, 2022.

\bibitem{Yudovich1}
V.~I. Judovi\v{c}.
\newblock Non-stationary flows of an ideal incompressible fluid.
\newblock {\em \v{Z}. Vy\v{c}isl. Mat i Mat. Fiz.}, 3:1032--1066, 1963.

\bibitem{K72}
Tosio Kato.
\newblock Nonstationary flows of viscous and ideal fluids in {${\bf R}^{3}$}.
\newblock {\em J. Functional Analysis}, 9:296--305, 1972.

\bibitem{M07}
Nader Masmoudi.
\newblock Remarks about the inviscid limit of the {N}avier-{S}tokes system.
\newblock {\em Comm. Math. Phys.}, 270(3):777--788, 2007.

\bibitem{HCE}
Helena~J. Nussenzveig~Lopes, Christian Seis, and Emil Wiedemann.
\newblock On the vanishing viscosity limit for 2{D} incompressible flows with
  unbounded vorticity.
\newblock {\em Nonlinearity}, 34(5):3112--3121, 2021.

\bibitem{S71}
H.~S.~G. Swann.
\newblock The convergence with vanishing viscosity of nonstationary
  {N}avier-{S}tokes flow to ideal flow in {$R_{3}$}.
\newblock {\em Trans. Amer. Math. Soc.}, 157:373--397, 1971.

\end{thebibliography}

\end{document}